\documentclass[12pt]{amsart}

\usepackage{comment}

\usepackage{hyperref}
\usepackage{graphicx}
\usepackage{amscd}
\usepackage{amsmath}
\usepackage{amsfonts}
\usepackage{bbm}
\usepackage{bbm}
\usepackage{amssymb}
\usepackage{verbatim}
\usepackage{comment}

\usepackage{color}

\usepackage{subfigure}

\definecolor{trp}{rgb}{1,1,1}

\definecolor{red}{rgb}{1,0,.2}

\definecolor{green}{rgb}{0.0,0.4,0.2}

\definecolor{blue}{rgb}{0,0,1}

\newcommand{\new}[1]{}
\newcommand*{\La}{\Lambda}

\newcommand*{\un}[1]{\underline{#1}}

\newcommand*{\be}{\begin{equation}}
\newcommand*{\ee}{\end{equation}}
\newcommand*{\ba}{\begin{aligned}}
\newcommand*{\ea}{\end{aligned}}
\newcommand*{\barr}{\begin{array}{c}}
\newcommand*{\earr}{\end{array}}
\newcommand*{\Ev}{{\mathbb{ E}}}
\newcommand*{\Pv}{{\mathbb{ P}}}

\newcommand*{\ind}{\mathbbm{1}}

\newtheorem{tm}{Theorem}[section]
\newtheorem{lm}[tm]{Lemma}
\newtheorem{df}[tm]{Definition}
\newtheorem{rem}[tm]{Remark}
\newtheorem{ex}[tm]{Example}
\newtheorem{cl}[tm]{Claim}
\newtheorem{cor}[tm]{Corollary}

\begin{document}

\parindent0pt

\title[Graphs generated by fractals]
{Generating hierarchial scale free graphs from fractals}

\author{J\'ulia Komj\'athy}
\address{J\'ulia Komj\'athy, Institute of Mathematics, Technical
University of Budapest, H-1529 B.O.box 91, Hungary
\tt{komyju@math.bme.hu}}

\author{K\'{a}roly Simon}
\address{K\'{a}roly Simon, Institute of Mathematics, Technical
University of Budapest, H-1529 B.O.box 91, Hungary
\tt{simonk@math.bme.hu}}

 \thanks{
\indent The research was supported by the NKTH OTKA grant \# 7778}

%\tableofcontents
\begin{abstract}

 Motivated by the hierarchial network model of E. Rav\-asz, A.-L. Barab\'asi, and T. Vicsek \cite{BRV} and \cite{AB}, we introduce deterministic scale-free networks derived from a graph directed self-similar fractal $\Lambda $.
 With  rigorous mathematical results we verify
that our model captures some of the most important features of many real networks: the scale free  and the high clustering properties. We also prove that the diameter is the logarithm of the size of the system.  Using our (deterministic) fractal $\Lambda $ we generate random graph sequence sharing  similar properties.

\end{abstract}
%\date{\today}

\maketitle

%%%%%%%%%%%%%%%%%%%%%%%%%%
%%%%%%%%%%%%%%%%%%%%%%%%%%%%%
%\color{white}
\section{Introduction}
In the last two decades  there have been a considerable  amount of attention paid to the study of complex networks like the World Wide Web, social networks, or biological networks. This resulted in the construction of numerous network models, see e.g. \cite{AB1}, \cite{JKK}, \cite{DM}, \cite{Bollobas}, \cite{LovSzeg} \cite{BJR}.  Most of them use a version of preferential attachment and are of probabilistic nature.
A completely different approach was initiated by
Barab\'asi, Ravasz, and Vicsek \cite{BRV}. They introduced deterministic network models generated by a method which is common in constructing fractals. Their model exhibits hierarchical structure and
the degree sequence obeys power law decay.
To model also the clustering behavior of real networks, Ravasz and Barab\'asi \cite{AB} developed the original model so that their deterministic network model preserved the same power law decay and has
similar clustering behavior to many real networks. Namely, the average local clustering coefficient is independent of the size of the network and the local clustering coefficient decays inversely proportional to the degree of the node.

In this paper we generalize both of the models above. Starting from  an arbitrary initial bipartite graph $G$ on $N$ vertices, we construct a hierarchical sequence of deterministic graphs $G_n$. Namely, $V(G_n)$, the set of vertices  of $G_n$ is $\left\{0,1,\dots ,N-1\right\}^n$.  To construct $G_n$ from $G_{n-1}$, we take $N$ identical copies of $G_{n-1}$, each of them identified with a vertex of $G$.  Then we connect these components in a complicated way described in (\ref{edgerule}). In this way, $G_n$ contains $N^{n-1}$ copies of $G_1$, which are connected in a hierarchical manner, see Figures \ref{G12}, \ref{G3} and \ref{19} for two examples.

There are no triangles in $G_n$. Hence, in order to model the clustering properties of many real networks, we need to extend the set of edges of our graph sequence to destroy the bipartite property. Motivated by \cite{AB}, we add some additional edges to $G_1$ to obtain the (no longer bipartite) graph $\widehat G_1$. Then we build up the graph sequence $\widehat G_n$ as follows: $\widehat G_n$ consist of $N^{n-1}$ copies of $\widehat G_1$, which copies are connected to each other in the same way as they were in $G_n$. So, $\widehat G_n$ and $G_n$ have the same vertex set and their edges only differ at the lowest hierarchical level, that is, within the $N^{n-1}$ copies of $G_1$ and $\widehat G_1$, see Figures \ref{19} and \ref{20}. We give a rigorous proof of the fact that the average local clustering coefficient of $\widehat G_n$ does not depend on the size and the local clustering coefficient of a node with degree $k$ is of order $1/k$.

The embedding of the adjacency matrix of the graph sequence $G_n$ is carried out as follows:
A vertex $\un x=(x_1\dots x_n)$ is identified with the corresponding $N$-adic interval $I_{\un x}$ (see (\ref{15})).
$\Lambda _n$ is the union of those $N^{-n}\times N^{-n}$ squares  $I_{\un x}\times I_{\un y}$
for which the vertices $\un x, \un y$ are connected by an edge in $G_n$. So, $\Lambda_n$ is the most straightforward embedding of the adjacency matrix of $G_n$ into the unit square.
 $\Lambda _n$ turns out to be a nested sequence of compact sets, which can be considered as the $n$-th approximation of a graph directed self-similar fractal $\Lambda $ on the plane, see Figure \ref{La13}.
We discuss connection between the graph theoretical properties of $G_n$ and  properties of the limiting fractal $\La$.

Furthermore, using $\Lambda$ we generate a random graph sequence $G_n^{\text{r}}$ in a way which was
inspired by the $W$-random graphs introduced  by Lov\'asz and  Szegedy \cite{LovSzeg}. See also  Diaconis, Janson \cite{DJ}, which paper contains a list of corresponding references. We show that the degree sequence has power law decay with the same exponent as the deterministic graph sequence $G_n$. Thus we can define a random graph sequence with a prescribed power law decay in a given range. Bollob'as, Janson and Riordan \cite{BJR} considered inhomogeneous random graphs generated by a kernel. Our model is not covered by their construction, since $\La$ is a fractal set of zero two dimensional Lebesgue measure.

\bigskip

The paper is organized as follows: In Section \ref{detmodel} we define the deterministic model and the associated fractal set $\La$. In Section \ref{prop},  we verify the scale free property of $G_n$ (Theorem \ref{16}). We compare the Hausdorff dimension of $\La$ to the power law exponent of the degree sequence of $G_n$.
 Our next result is that both of the diameter of $G_n$ and the average length of shortest path between two vertices
 are of order of the logarithm of the size of $G_n$ (Corollary \ref{diam} and Theorem \ref{averst}).
In Section \ref{clust} we prove the above mentioned properties of the clustering coefficient of $\widehat G_n$ (Theorem \ref{clusteraver} and \ref{clustertetel}). In Section \ref{rand} we describe the randomized model, and in Section \ref{proprand} we prove that the model exhibits the same power law decay as the corresponding deterministic version.

\section{Deterministic model}\label{detmodel}

The model was motivated by the hierarchical graph sequence model in \cite{BRV}, and is given as follows.

\subsection{Description of the model}%\label{descript}

 Let $G$, our \emph{base graph}, be any labeled bipartite graph on the vertex set $\Sigma_1=\left\{0,\dots ,N-1\right\}$. We partition $\Sigma _1$ into the non-empty sets $V_1,V_2$ and one of the end points of any edge is in $V_1$, and the other is in $V_2$. We write $n_i:=|V_i|$, $i=1,2$ for the cardinality of $V_i$.  The edge set of $G$ is denoted by $E(G)$. If the pair $x,y \in \Sigma_1$ is connected by an edge, then this edge is denoted by ${ x \choose y }$, since this notation makes it convenient to follow the labels of the vertices along a path.

 Now we define our graph sequence $\left\{ G_n\right\}_{n\in \mathbb N}$ generated by the base graph $G$.

 The vertex set is $\Sigma_n= \{ (x_1 x_2\dots x_n): x_i \in \Sigma_1\}$, all words of length $n$ above the alphabet  $\Sigma_1$. To be able to define the edge set, we need some further definitions.

  %%%%%%%%%%%%%%%%%%%%%%%%%%%%%

  \begin{df}\label{typ} \

\begin{enumerate}
  \item We assign a type to each element of $\Sigma_1$. Namely,
$$
\text{typ}(x)=
\left\{
  \begin{array}{ll}
    1, & \hbox{if $x \in V_1$;} \\
    2, & \hbox{if $x \in V_2$.}
  \end{array}
\right.
$$
  \item We define the {\bf type} of a word $\un z=(z_1z_2\dots z_n )\in \Sigma _n$ as follows: if all the elements $z_j, j=1,\dots, n$ of $\un z$ fall in the same $V_i$, $i=1,2$ then $\text{typ}(\un z)$ the type of $\un z$ is  $i$. Otherwise $\text{typ}(\un z):=0$.

\item  For  $\un x= (x_1\dots x_n), \un y=(y_1\dots y_n) \in \Sigma_n$ we denote the {\bf common prefix} by
\[  \un x \wedge \un y = (z_1 \dots z_k) \text{ s.t. } x_i=y_i=z_i,  \forall i=0,\dots, k \text{ and } x_{k+1}\neq y_{k+1}. \]
\item Given $\un x= (x_1\dots x_n), \un y=(y_1\dots y_n) \in \Sigma_n$,  the {\bf postfixes} $\tilde {\un x}, \tilde {\un y} \in \Sigma_{n-|\un x \wedge \un y|}$ are determined by
    \[\un x = (\un x \wedge \un y) \tilde {\un x},\  \un y = (\un x \wedge \un y) \un{\tilde y}, \] where the concatenation of the words $\un a,\un b$ is denoted by $\un a \un b$.
\end{enumerate}\end{df}
  %%%%%%%%%%%%%%%%%%%%%%%%
Now we can define the edge set $E(G_n)$. Two vertices $\un x$ and $\un y$ in $G_n$ are connected by an edge if and only if the following assumptions hold:
\begin{description}
  \item[(a)]  One of the postfixes $\un {\tilde x}, \un {\tilde y}$ is of type $1$, the other is of type $2$,
  \item[(b)]  for each $i>|x\wedge y|$, the coordinate pair ${ x_i \choose y_i}$ forms an edge in $G$.
\end{description}
That is, $E(G_n)\subset \Sigma_n\times \Sigma_n$:
\begin{eqnarray}\label{edgerule}
% \nonumber to remove numbering (before each equation)
\nonumber E(G_n)& =& \Bigg\{\! { \un x\choose \un y}\Big|\Big. \un x=\un y \mbox{ or } \\
   &&
\{ \text{typ}(\un {\tilde x}),\text{typ}(\tilde{\un y})\} =\{ 1,2\}, \forall |\un x\wedge \un y|<i\le n,\!{x_i\choose y_i}\!\in\!E(G)
\Bigg\}
\end{eqnarray}
\begin{rem}%\label{selfloopa}
Note that we artificially added all loops to the (otherwise bipartite) graph sequence $G_n$, implying easier calculations later without loss of the important properties. In particular, $G_1$ differs from $G$ only in the loops.
\end{rem}
\begin{rem}[Hierarchical structure of $G_n$]%\label{hierarch}

For every initial digit $x \in \{0,1, \dots, N-1\}$, consider the set $W_x$ of vertices $(x_1 \dots x_n)$ of $G_n$ with $x_1=x$. Then the induced subgraph on $W_x$ is identical to $G_{n-1}$.
\end{rem}

We write $\deg_n(\un x)$ for the \emph{degree of a vertex} in $G_n$, including the loop which increases the degree by $2$. However, for an $x\in \Sigma_1$, $\deg x$ denotes degree of $x$ in $G$. In particular $\deg_1(x)=\deg(x)+2$.
In what follows, we will frequently use $\ell(\un x)$, the length of the longest block from backwards in $\un x$ which has a nonzero type,
\be \label{ell}
\ell (\un x):= \max_{i\in \mathbb N}\left\{ \text{typ}(x_{n-i+1},\dots x_n) \in \{1,2\}\right\}
\ee

\begin{rem}\label{5}

The degree of a node $\un x\in \Sigma _n$
\[%\label{7}
  \deg_n(\un x)=2+S(\un x)\cdot \deg(x_n),
\]
where
\begin{eqnarray}\label{6}
%  to remove numbering (before each equation)
 \nonumber  S(\un x):\!\! &=&\!\! 1+\deg(x_{n-1})+\cdots +\deg(x_{n-1})\deg(x_{n-2})\cdots \deg(x_{n-\ell(\un x)+1}) \\
   &=&  \sum_{r=0}^{\ell(\un x)-1} \left(\prod_{j=1}^r \deg(x_{n-j})\right),
\end{eqnarray}
where the empty sum is meant to be $1$.
\end{rem}

The following two examples satisfy the requirements of our general model.
\begin{ex}[Cherry]\label{cherry}
Barab\'asi, Ravasz and Vicsek \cite{BRV} introduced the "cherry" model presented on Figures \ref{G12} and \ref{G3}: Let $V_1=\{1\}$ and $V_2=\{0,2 \}$, $E(G)=\left\{ (1,0), (1,2)\right\}$.
\end{ex}

\begin{ex}[Fan]\label{gkal} Our second example is  called "fan", and is defined on Figure \ref{19}. Note that here $|V_1|>1$.
\end{ex}

\begin{figure}
\centering
\subfigure[$G_1$ and $G_2$ with loops]{\label{G12}
\includegraphics[keepaspectratio,height=3cm]{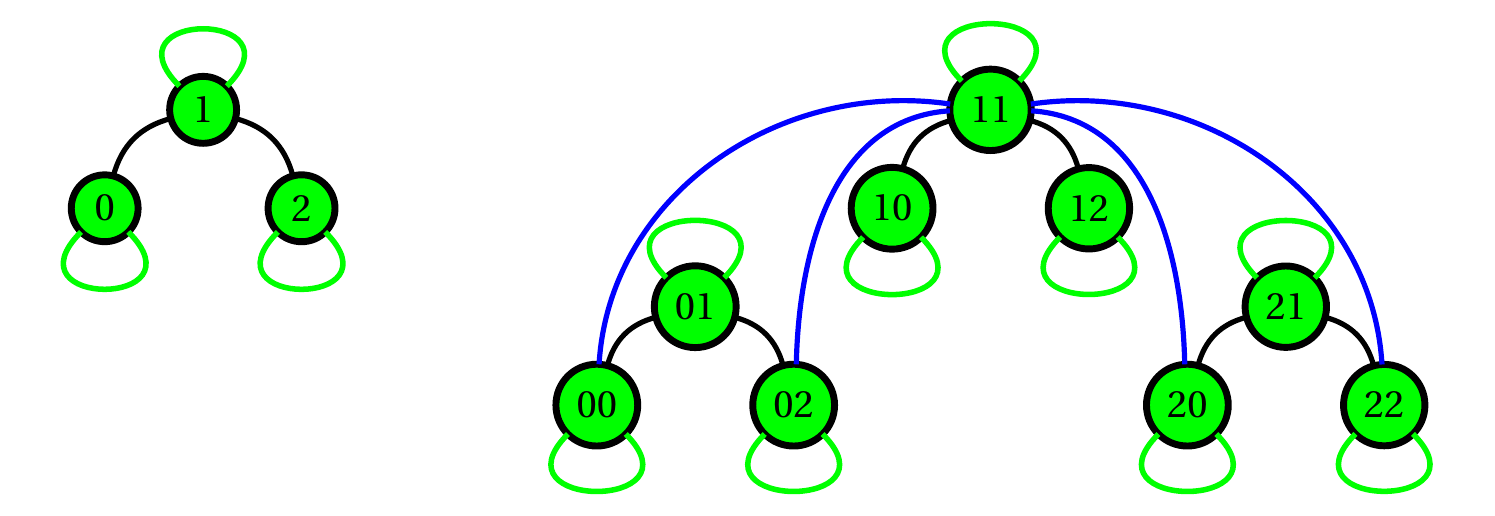}}
\subfigure[$G_3$]{\label{G3}
\includegraphics[keepaspectratio,width=12cm]{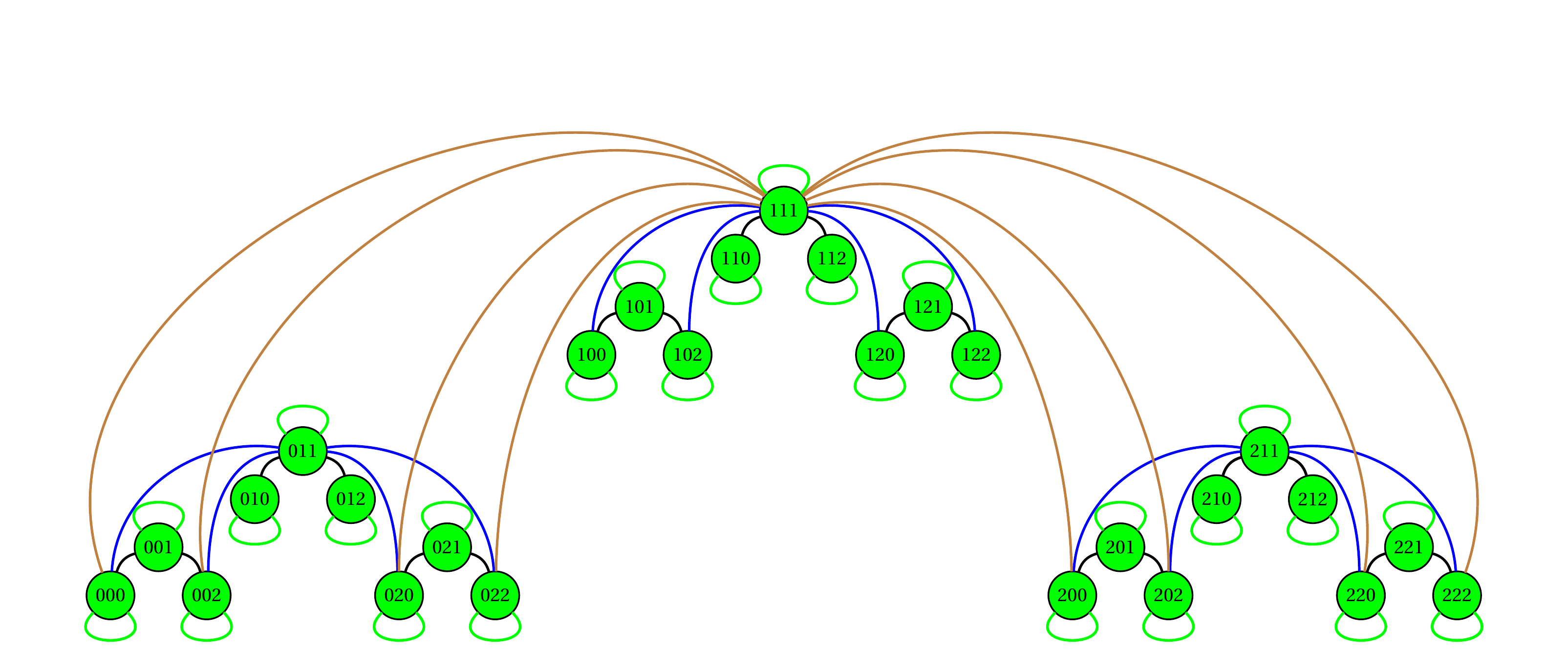}}

\subfigure[The sets $\Lambda _1,\Lambda _2,\Lambda _3$  ]{\label{La13}

\includegraphics[keepaspectratio,width=12cm]{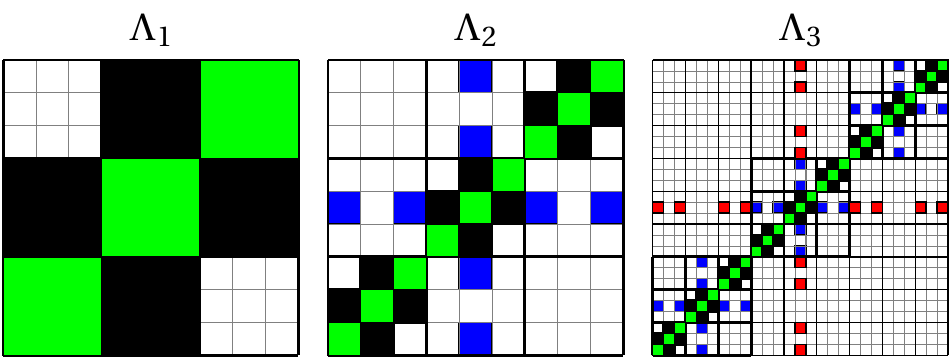}}

\caption{$G_1,G_2,G_3,\Lambda _1,\Lambda _2,\Lambda _3$  for the cherry Example \ref{cherry}}\label{21}

\end{figure}

\subsection{The embedding of the adjacency matrices into $[0,1]^2$}\label{sectionadjac}
In this Section, we investigate the  sequence of adjacency matrices corresponding to $\left\{G_n\right\}_{n\in \mathbb N}$. Roughly speaking, we will map them in the unit square, see Figure \ref{La13}.

To represent the adjacency matrix of $G_n$ as a subset of the unit square,  first partition $[0,1]^2$ into $N^{2n}$ congruent boxes, i.e. divide $[0,1]$ into equal subintervals of length $\frac{1}{N^n}$, corresponding to the first $n$ digits of the $N$-adic expansion of elements of $[0,1]$:
\be\label{15} I_{x_1\dots x_n}= \left[\sum_{r=1}^n \frac{x_r}{N^r}, \sum_{r=1}^n \frac{x_r}{N^r}+\frac{1}{N^n}\right], \forall (x_1 \dots x_n) \in \Sigma_n.\ee
We partition $[0,1]^2$ with the corresponding level-$n$ squares:
\be \label{Qxy} Q_{{\un x \choose \un y}}:=I_{\un x}\times I_{\un y}, \quad {\un x\choose \un y} \in \Sigma_n\times \Sigma_n.\ee
A natural embedding of the adjacency matrix of $G_n$ in the unit square is as follows:
\be\label{Lambdan}
\Lambda_n(a,b) := \left\{
                    \begin{array}{ll}
                      1, & \hbox{if $(a,b) \in Q_{{\un x \choose \un y}}, {\un x \choose \un y} \in E(G_n)$;} \\
                      0, & \hbox{otherwise.}
                    \end{array}
                  \right.
\ee
That is,
\[\Lambda_n(a,b) = \sum_{\un x, \un y \in \Sigma_n \atop{\un x \choose \un y } \in E(G_n)}  \ind_{ Q_{{\un x \choose \un y}}}(a,b). \]

We write $\Lambda_n$ for the support of the function $\Lambda_n(a,b)$, see Figure \ref{La13}. Observe that $\Lambda_n$ is a compact set and
$\Lambda_{n+1} \subset \Lambda_{n}$ holds for all $n$. So we can define the non-empty compact set
\be \label{Lambda} \Lambda:= \bigcap\limits_{n=1}^\infty \La_n. \ee
Clearly,
\[ \ind_\La(a,b)=\lim\limits_{n\to\infty}\La_n(a,b).\]

\begin{rem}

This representation obviously depends on the labeling of the graph $G$. For an arbitrary permutation $\pi$ of $\{0,\dots, N-1\}$, the corresponding representation of $G_n$ is denoted by $\Lambda^\pi_n(a,b)$. The relation between these two representations is given by the formula  \[\ba \Lambda^\pi_n(a,b)&=\Lambda_n(\varphi_{\pi^{-1}}(a),\varphi_{\pi^{-1}}(b)), \mbox{ and}\\
\ind_{\Lambda^\pi  }(a,b)&=\ind_{\Lambda }(\varphi_{\pi  ^{-1}}(a),\varphi _{\pi ^{-1}}(b)),\ea\]
where the measurable function $\varphi_\pi(x):[0,1]\to [0,1]$ is defined by
$$
\varphi_\pi \left(\sum_{i=1}^\infty \frac{x_i}{N^i}\right)=\sum_{i=1}^\infty \frac{\pi(x_i)}{N^i}.
 $$
\end{rem}

%\noindent {\bf Notation. }A complete directed graph on vertex set $V$ with $|V|=k$ is denoted by $\overrightarrow{K}_k(V)$.
%Similarly we write $\overrightarrow{B}_{k,l}(V,V')$ for the complete bipartite graph where we direct all the edges from $V$ to $V'$ and $|V|=k$ and $|V'|=l$.
%Let $V$ and $V'$ be disjoint set of vertices with cardinality $k$ and $\ell$, respectively. We define a directed bipartite graph $\overrightarrow{B}_{k,\ell}(V,V')$ with
%edge set $V\times V'$ such that all edges are directed from $V$ to $V'$.

%Furthermore, we will see that $\La_n$ is the $n$-th approximation of $\La$.

%The sequence $\Lambda_n$ form a convergent graph-directed self-similar set in $[0,1]^2$.
\subsection{Graph-directed structure of $\La$}%\label{graphdir}
Now we prove that  the limit $\La$ (defined in (\ref{Lambda})) can be considered as the attractor
of a not irreducible graph-directed self-similar iterated
 function system, (for the definition see \cite{Falconer}),
  with the directed graph $\mathcal G$ defined below.

\begin{df}%\label{GD}
The {\bf vertex set} $V(\mathcal G)$ is partitioned into three subsets:
\be\ba \label{dd1221} V_{dd}&=\left\{ {z\choose z}, z\in \Sigma_1\right\}\\
 V_{12}&=\left\{ {x\choose y}, x\in V_1, y \in V_2 \right\} \\
  V_{21}&=\left\{ {x\choose y}, x\in V_2, y \in V_1\right\}.
\ea
\ee
Then
\[ V(\mathcal G)=  V_{dd}\cup V_{12}\cup V_{21}.\]
The {\bf set of directed edges} $ E (\mathcal G)$ of $\mathcal G$ is as follows:
First we connect all vertices in both directions within each of the three sets $V_{dd}$, $V_{12}$ and $V_{21}$ (loops included). Then there is an outgoing edge for each vertex in $V_{dd}$ to all vertices in $V_{12}$ and $V_{21}$.
%\[ \mathcal E(\mathcal G)= \overrightarrow{K_N}(V_{dd}) \cup \overrightarrow{K_{|E|}}(V_{12}) \cup \overrightarrow{K}_{|E|}(V_{21}) \cup \overrightarrow{B}_{N,|E|}(V_{dd},V_{12})\cup \overrightarrow{B}_{N,|E|}(V_{dd},V_{21})\]

For every directed edge $e=(v_1,v_2)\in E(\mathcal G) $ we define a homothety:
\be\label{fe} f_e: \ Q_{v_2}\to Q_{v_1},\  f_e(a,b):= \frac1N(a,b) +\frac1N( x_1, y_1),  \ \mbox{with } v_i= { x_i \choose y_i},\ee
where $Q_{ v}:=Q_{ x\choose y }$ is the level-1 square for $ v={ x\choose y } \in V(\mathcal G)$.
\end{df}
The graph $\mathcal G$ corresponding to the graph sequence in the "cherry" example is given by Figure \ref{calG}.

\begin{figure}[ht!]
  % Requires \usepackage{graphicx}
\begin{center}
\includegraphics[keepaspectratio,height=9cm]{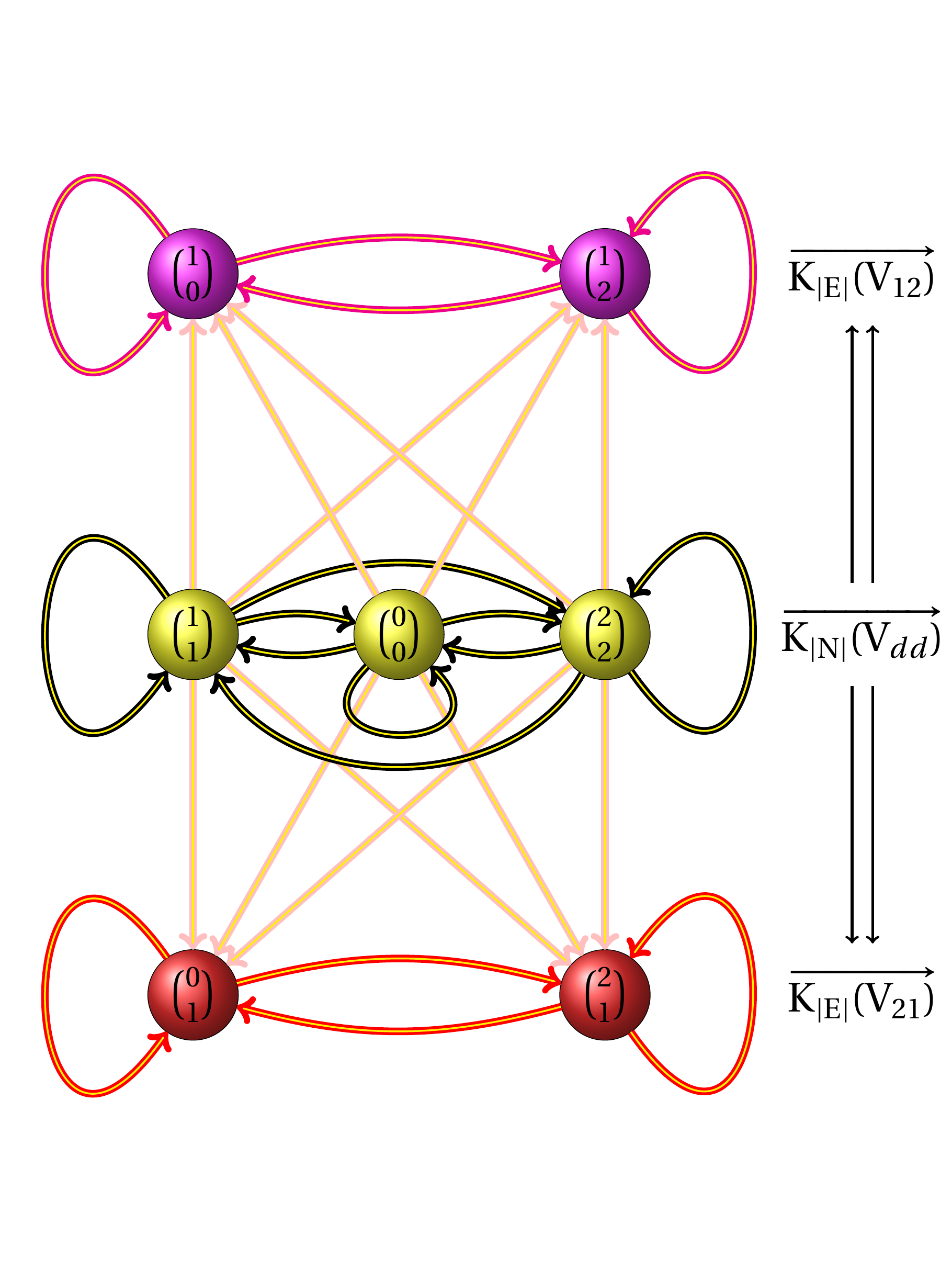}\\
  \caption{The graph $\mathcal{G}$ for the "cherry", Example \ref{cherry}.}\label{calG}
\end{center}
\end{figure}

In general, $\mathcal G$ is given by the schematic picture on the right hand side of Figure \ref{calG}, where the double arrow in between the complete directed graphs $\overrightarrow{K_{.}(V_{..})} $ illustrates that we connect all pairs of vertices in the given direction.

Let $\mathcal P_n$ be the set of all paths of length $n$ in $\mathcal G$, i.e.
\[ \mathcal P_n:= \left\{ \un v=(v_1\dots v_n)| \forall \ 1\le i <n \ (v_i, v_{i+1}) \in E(\mathcal G) \right\}. \]
For a  $\un v=(v_1\dots v_n)={x_1\dots x_n\choose y_1 \dots y_n } \in \mathcal  P_n$ it immediately follows from definitions (\ref{Qxy}) and (\ref{fe}) that
\be \label{Qiter} Q_{\un v}=  f_{\un v}\left([0,1]^2\right)= I_{ x_1\dots x_n}\times I_{y_1\dots y_n},\ee where
\be \ba \label{fv} f_{\un v}(.):&=f_{(v_1,v_2)}\circ \dots \circ f_{(v_{n-1},v_n)}(.) \text{ if } n\ge2,\\
 f_{v}(a,b):&=\frac1N (a,b)+\frac1N (x,y), \text{ if } n=1, \ v={x \choose y}.
 \ea
\ee
The key observation of connecting $\mathcal G$ to the graph sequence $G_n$ is the following:
\begin{cl}\label{ekvipath}For all $n$ we have
\[ \  E(G_n)=\mathcal P_n. \]
\end{cl}
\begin{proof}
 Let
  $\un v=(v_1\dots v_n)={a_1\dots a_n\choose b_1 \dots b_n }\in \Sigma _n\times \Sigma _n$, thus $\un a=(a_1\dots a_n)$ and $\un b=(b_1\dots b_n)$ are vertices in $G_n$. First we assume that $\un v\in E(G_n)$. Observe that by (\ref{edgerule}), ${ a_i\choose b_i}$ are vertices in $\mathcal G$. We would like to prove that the sequence
  \be \label{seq}  {a_1\choose b_1}\dots {a_n\choose b_n} \in \mathcal P_n.\ee

  If $k:=|\un a\wedge \un b|\ge 1$, then for
  $ i\leq k$, $a_i=b_i $ holds, thus the sequence of points ${a_1\choose b_1}\dots {a_k\choose b_k}$ forms a path in $\overrightarrow{K_{|N|}(V_{dd})}$. By (\ref{edgerule}), the pairs ${a_{k+1}\choose b_{k+1}},\dots, {a_n\choose b_n}$ are all edges in $G$ thus vertices in $\mathcal G$. Furthermore, either they all belong to $V_{12}$ or they are all contained in $V_{21}$, see (\ref{dd1221}). This implies that this postfix also forms a path in $\overrightarrow{K_{|N|}(V_{12})}$ or in $\overrightarrow{K_{|N|}(V_{21})}$. By definition of $E(\mathcal G)$, $\left({a_k\choose b_k}, {a_{k+1}\choose b_{k+1}}\right)$ is
 an edge in $\mathcal G$, so ${a_1\choose b_1}\dots {a_n\choose b_n}$ is a path in $\mathcal G$. If $k=0$ then the whole path is contained either in  $V_{12}$ or in $V_{21}$.
 This completes the proof of (\ref{seq}).

On the other hand, if $ {a_1\choose b_1}\dots {a_n\choose b_n}$ is a path of length $n$ in $\mathcal G$, then we claim that for $\un a=(a_1 \dots a_n), \un b=(b_1\dots b_n)\in V(G_n)$
\[ (\un a, \un b)\in E(G_n).\]The proof is very similar to the previous one. \end{proof}
In this way we can characterize $\La_n$ as follows:
\begin{cor}\label{ekvilan}
\[ \La_n= \bigcup_{\un v \in \mathcal P_n} Q_{\un v}=\bigcup_{\un v \in \mathcal P_n} f_{\un v}\left([0,1]^2\right) .\]
\end{cor}

\begin{proof}
Immediately follows from (\ref{Lambdan}) and (\ref{Qiter}) and the assertion of the Claim \ref{ekvipath}.
\end{proof}
Let us define \[ \mathcal P_\infty:= \{ \un v=(v_1 v_2 \dots)| \forall i\in \mathbb N, \ (v_i, v_{i+1}) \in E(\mathcal G) \}.\]
Now for every $\un v\in \mathcal P_\infty$ we have $\bigcap\limits_{n=1}^\infty Q_{(v_1 \dots v_n)} $ is a point in $[0,1]^2$, which will be denoted by $\Pi_{\un v}$. That is,
\[ \Pi: \mathcal P_{\infty} \to [0,1]^2, \ \Pi(\un v):= \bigcap\limits_{n=1}^\infty Q_{(v_1 \dots v_n)}=\lim\limits_{n\to\infty} f_{v_1\dots v_n}(0,0).\]
It is an immediate consequence of Corollary \ref{ekvilan}, that
\be\label{lapi} \Pi(\mathcal P_{\infty})= \La, \text{ i.e. } \La= \bigcup_{\un v \in \mathcal P_\infty} \Pi_{\un v}.\ee
  This means that $\La_n$, the embedded adjacency matrix of $G_n$, can be considered as the $n$-th approximation of the fractal set $\La$.

In this way we coded the elements of $\La$ by the elements of $\mathcal P_\infty$. This coding is not $1-1$ for the same reason as the $N$-adic expansion is not $1-1$. However, if neither of the two coordinates of a point $(a, b) \in \La$ are $N$-adic rational numbers, then (a,b) has a unique code.

\subsection{Fractal geometric characterisation of $\La$.}%\label{frac}

For notational convenience we define the set of finite words above the alphabet $V_{dd}$ (including the empty word as well):
\[ V_{dd}^*:=\left\{\un v |\  \exists n\in \mathbb N\cup \left\{0\right\}, \un v = (v_1 \dots v_n)  \text{ and } v_i\in V_{dd}\right\}.\]

The three subgraphs $\overrightarrow{K_{|E|}(V_{12})}$, $\overrightarrow{K_{|E|}(V_{21})}$ and $\overrightarrow{K_{|E|}(V_{dd})}$ of $\mathcal G $ are complete directed graphs. We consider the three corresponding self-similar iterated function systems (IFS):
\[\ba \mathcal{F}_{dd}:&=\left\{f_{v}\right\}_{v\in V_{dd}},\\
\mathcal{F}_{12}:&=\left\{f_{v}\right\}_{v\in V_{12}}, \\ \mathcal{F}_{21}:&=\left\{f_{v}\right\}_{v\in V_{21}}, \ea \] where the functions $f_{v}, v\in V(\mathcal G)$ were defined in (\ref{fv}). The attractors of these IFS-s (see \cite[p.30]{Falconer}) are the unique nonempty compact sets satisfying
\be \ba\label{attraktor}
\La_{dd}&:= \bigcup_{v \in V_{dd}} f_{v}(\La_{dd})=\left\{ \Pi(\un v)| \un v=(v_1,v_2 \dots ) \text{ and }  v_i \in V_{dd} \right \}\\
\La_{12}&:= \bigcup_{v \in V_{12}} f_{v}(\La_{12})=\left\{ \Pi(\un v)| \un v=(v_1,v_2 \dots ) \text{ and }  v_i \in V_{12} \right \}\\
\La_{21}&:= \bigcup_{v \in V_{21}} f_{v}(\La_{21})=\left\{ \Pi(\un v)| \un v=(v_1,v_2 \dots ) \text{ and }  v_i \in V_{21} \right \}.
\ea
\ee
The Open Set Condition (see e.g. \cite[p.35]{Falconer}) holds for these IFS-s, so we can easily compute the Hausdorff-dimension of the attractors. Clearly, $\La_{dd}$ is the diagonal of the unit square.

Now we  prove that $\La$ is a countable union of homothetic copies of these attractors.
\begin{tm}\label{tet1}
\[ \La= \underbrace{\mathrm{ Diag }}_{\La_{dd}}\cup \bigcup_{\un v \in V_{dd}^*} \left(f_{\un v}(\La_{12})\cup f_{\un v}(\La_{21})\right),\]
where $\mathrm{ Diag }=\left\{ (x,x): x\in [0,1] \right\}$.
\end{tm}

\begin{rem} Observe that $\La_{21}$ is the image of $\La_{12}$ by the reflection through the diagonal, hence $\La$ is symmetric to the diagonal. The same is true for the $n$-th approximation $\La_n$ of $\La$. This can be seen immediately by using the embedded adjacency matrix characterization of $\La_n$.
\end{rem}
\begin{proof}[Proof of Theorem \ref{tet1}]
We start by showing that
\be\label{subset}\La \subset \mathrm{ Diag }\cup \bigcup_{\un v \in V_{dd}^*} \left(f_{\un v}(\La_{12})\cup f_{\un v}(\La_{21})\right).\ee
Pick an arbitrary point  $(a,b)\in \La$. As a consequence of (\ref{lapi}) there exists a $\un v=(v_1 v_2 \dots)\in \mathcal P_\infty$ such that $\Pi(\un v)=(a,b)$. Let $k:=\max\left\{\ell:v_\ell\in \Lambda _{dd}\right\}$. We distinguish three cases: $k=0$, $k=\infty $ or $0<k<\infty$. Mind that for all $i\le k, v_i\in V_{dd}$ since once the path left the component $V_{dd}$, there is no way to return. Since $V_{12}$ and $V_{21}$ are closed, for $k< \infty$ all $v_i, i>k$  are in the same component $V_{12}$ or $V_{21}$.
\begin{description}
  \item[Case $k=0$] Clearly either all $v_i$ are in  $V_{12}$ or in $V_{21}$, so  $\Pi(\un v)\in \La_{12}\cup \La_{21}$.
  \item[Case $k=\infty $] For the same reason,
  $\Pi(\un v)=\lim\limits_{n\to\infty} f_{v_1\dots v_n}(0,0)\in \La_{dd}=\mathrm{ Diag }$. This is so because $f_{v_1\dots v_n}(0,0)$ is in the $\frac{1}{N^n}$ neighborhood of the diagonal $\left\{(x,x):x\in [0,1]\right\}$.
  \item[Case $0<k<\infty $] Let $\un v_k= (v_1 \dots v_k)$. For symmetry, without loss of generality we may assume that $v_{k+1}\in V_{12}$. As in the first case, we can see that for $\un w:= (v_{k+1}v_{k+2}\dots )$, $\Pi(\un w)\in \La_{12}$. Hence $\Pi(\un v)=f_{\un v_k}(\Pi_{\un w})\in f_{\un v_k}(\La_{12})$.
\end{description}
Now we have verified (\ref{subset}). To prove the opposite direction, that is
\be \label{supset}\La \supset\mathrm{ Diag }\cup \bigcup_{\un v \in V_{dd}^*} \left(f_{\un v}(\La_{12})\cup f_{\un v}(\La_{21})\right),\ee
we will use the symbolic representation of $\La$ given in (\ref{lapi}).

Pick an $x \in [0,1]$ and take the $N$-adic code $(x_1 x_2 \dots)$ of $x$. That is, $x=\sum\limits_{n=1}^{\infty}\frac{x_i}{N^i}$, $x_i \in \{0,\dots, N-1\}$. Then \[ \un v:=\Bigg(\underbrace{{x_1\choose x_1}}_{v_1},\underbrace{{x_2\choose x_2}}_{v_2},\dots\Bigg) \in \mathcal P_\infty, \]
it is easy to see that $\Pi(\un v)=(x,x)$. So by (\ref{lapi}), $(x,x) \in \La$.

Now we assume that $(a,b)\in \bigcup\limits_{\un v \in V_{dd}^*} \left(f_{\un v}(\La_{12})\cup f_{\un v}(\La_{21})\right)$. Without loss of generality we may further assume that $(a,b)\in f_{\un v}(\La_{12})$ for some $\un v \in V_{dd}^*$.  That is,
$ (a,b)=f_{\un v}(a',b')$ where $(a',b') \in \La_{12}$. By (\ref{attraktor}) there exists a
$\un w:=(w_1 w_2\dots ),\ w_i\in V_{12}$ such that $\Pi (\un w)=(a',b')$. In this way,
for the concatenation $\un t:=\un v\un w\in \mathcal{P}_\infty $ we have $(a,b)=\Pi (\un t)$ which implies $(a,b)\in \La$. This completes the proof of (\ref{supset}).
\end{proof}
\subsection{The same model without loops.}\label{selfloop}
Let $G'_n$ be the same graph as $G_n$ but without loops, i.e. $V(G'_n)= V(G_n)$ and $E(G'_n)\subset \Sigma_n\times \Sigma_n$:
\begin{eqnarray*}
 E(G'_n) = \Bigg\{ { \un x\choose \un y}&\Big|& \{ \text{typ}(\un {\tilde x}),\text{typ}(\tilde{\un y})\} =\{ 1,2\} \mbox{ and } \\
   &&\quad \forall |\un x\wedge \un y|<i\le n,\!{x_i\choose y_i}\!\in\!E(G)
\Bigg\}
\end{eqnarray*}
In this case $\La'_n= \La_n\!\setminus\!\text{Diag}_n  $, where $\text{Diag}_n$ is the union of the level $n$ squares that have nonempty intersection with the diagonal. The sequence $\La'_n$ is not a nested sequence of compact sets. However, it is easy to see that the characteristic function of $\La'_n$ tends to characteristic function of $\La\setminus\text{Diag}$. Further, $\La'_n$ tends to $\La$ in the Hausdorff metric, see \cite{Falconer}.

%We say that the equivalence class of $\La$ is the "limit" of the graph sequence $G_n$ as a labeled graph. This is similar to the limit \cite{LovSzeg}, but here the limit $W$ is of zero Lebesgue measure.

\section{Properties of the sequence $\left\{G_n\right\}$ and $\La$}\label{prop}
In this section we compute the degree distribution of $G_n$, and relate it to the Hausdorff dimension of $\La$. We also compute the length of the average shortest path in $G_n$. To get interesting result about the local clustering coefficient we need to modify our graph sequence $G_n$ in the line as it was done in \cite{AB}.

\subsection{Degree distribution of $\left\{ G_n\right\}$}\label{degdistr}
 Here we compute the degree distribution under the following regularity assumption on the base graph $G$:
\be \label{degassum} \tag{\textbf{A1}}\ba &\deg(x):=d_1, \quad \forall  x\in V_1 \\
&\max_{j \in V_2}\deg(y):=d_2 \le d_1-1, \quad \forall y\in V_2
\ea
\ee

\begin{figure}
\centering
\subfigure[$G$ on the left and $G_1$ on the right hand side. Here $V_1=\left\{2,4\right\}$ and $V_2=\left\{0,1,3,5\right\}$]{\label{14}
\includegraphics[keepaspectratio,height=3cm]{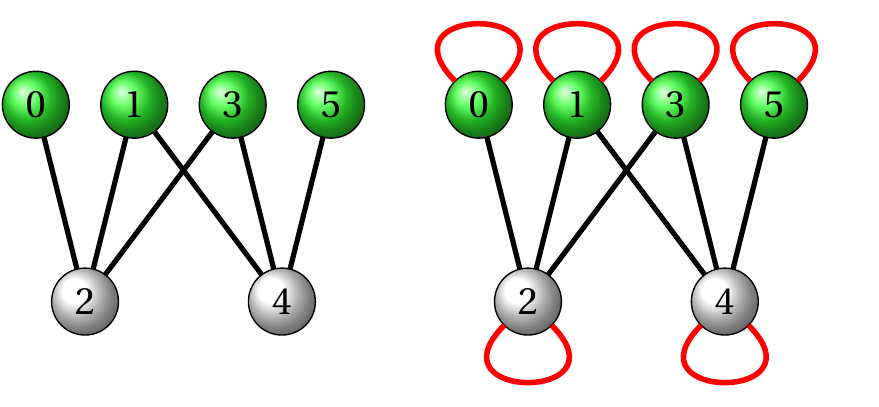}}

\subfigure[The graph $G_2$ (contains additionally all loops).]{\label{13}
\includegraphics[keepaspectratio,width=12cm]{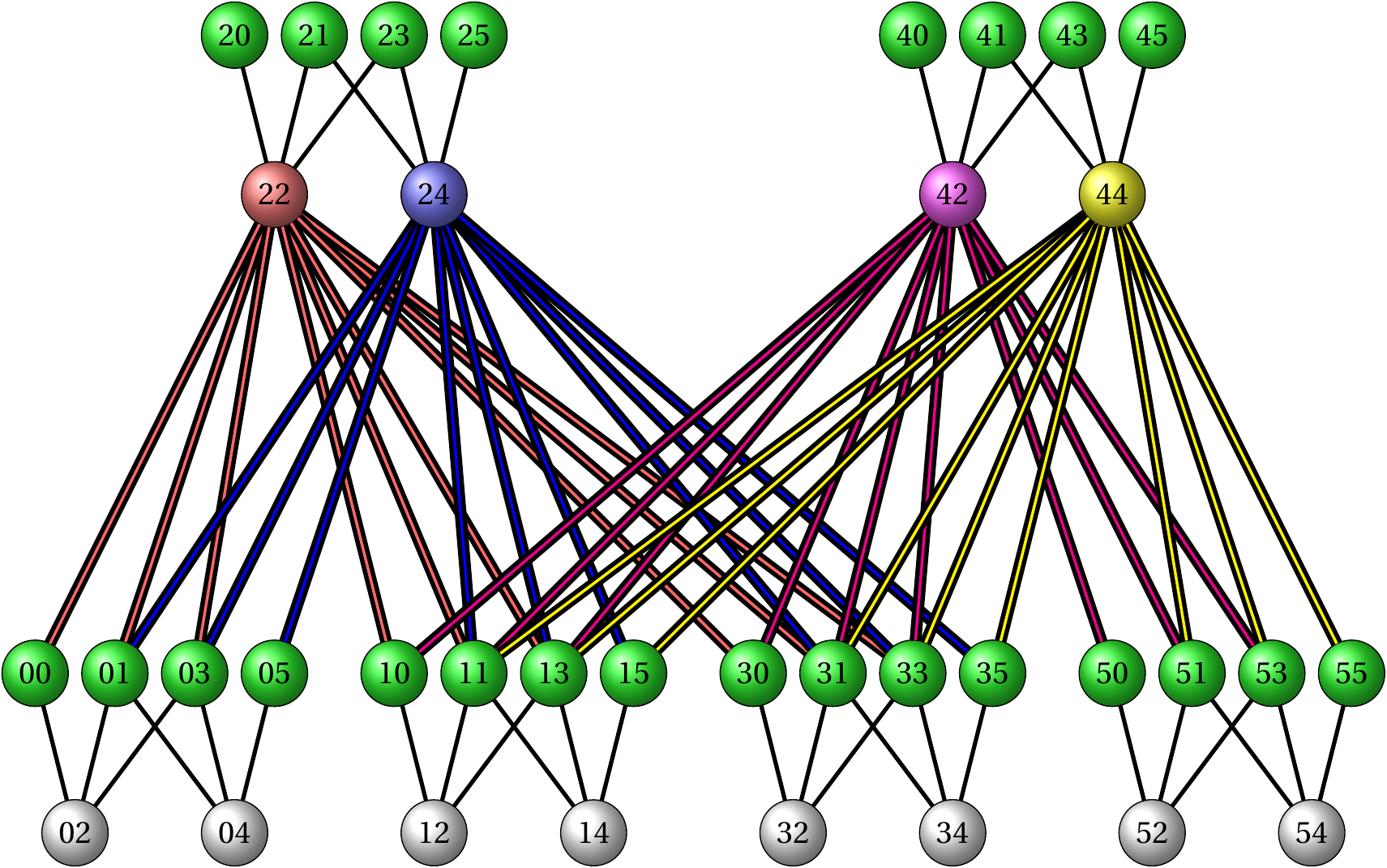}}
\caption{Example "fan".}\label{19}

\end{figure}

Recall that we defined $\ell(\un x)$ in (\ref{ell}) as the length of the longest block from backwards of the node $\un x$ such that the last $\ell(\un x)$ digits of $\un x$ belong to the same $V_i$. Put $\Sigma_n^i:=\left\{\un x\in \Sigma_n| x_n \in V_i\right\}, i=1,2$.
It follows from \ref{degassum} and Remark \ref{5} that the degree of a node  $\un x \in \Sigma_n^1$  is $\frac{d_ 1^{\ell(\un x)+1}-1}{d_1-1}+1$, and the number of such nodes with $\ell(\un x )=\ell$ is exactly $N^{n-\ell+1} \cdot n_2\cdot n_1^\ell.$

Under assumption \ref{degassum}, the decay of the degree distribution is determined by the set of high degree nodes denoted by
\[ HD_n:=\left\{\un x \in \Sigma_n^1| \deg_n(\un x ) > \max\limits_{\un y \in \Sigma_n^2} \deg_n(\un y) \right\}.\]
An equivalent characterisation of $HD_n$ is
\[ HD_n=\left\{\un x \in \Sigma_n^1| \ell(\un x)> \frac{1}{\log d_1}\max\left\{ (n+1) \log(d_2), \log n \right\}\right\}.\] This is so because  the degree of any $\un y\in \Sigma_n^2$ is at most $\max\{ d_2^{n+1},n\}$ .
The tail of the cumulative degree distribution is
\[  \Pv\left[\deg_n(\un X) > \frac{d_1^{\ell+1}-1}{d_1-1}+1\right]= \frac{n_1^{\ell+1} N^{n-\ell-1}}{N^n}= \left(\frac{n_1}{N}\right)^{\ell+1},\]
where $\un X$ is a uniformly chosen node of $G_n$. Mind that as long as $\ell<n$, this probability does not depend on $n$.
Writing $\widetilde F(t)=\Pv(\deg(\un X)> t)$ for the tail of the cumulative distribution function we get the power law decay
\[ \widetilde F(t)= t^{-\frac{\log(N/n_1)}{\log d_1}}\cdot c(d_1) \quad \text{ for }t= \frac{d_1^{\ell+1}-1}{d_1-1}.
 \]
So we have proved
\begin{tm}\label{16}
The degree distribution of the graph sequence $G_n$ satisfying assumption \ref{degassum}, has a power law decay with exponent
\be \label{expon}\tilde \gamma=\gamma-1= \frac{\log(N/n_1)}{\log d_1}.\ee
\end{tm}
 This implies that the largest decay $\gamma$ we can get in this family of models is  $1+\frac{\log 3}{\log 2}$,  and the maximum is attained at $n_1=1$ and $d_1=2=n_2$. This is exactly the graph sequence in Example \ref{cherry}, see Figures \ref{G12} and \ref{G3}.
  We will later see that the case $n_1=1$ is important in another sense as well, see Section \ref{Hausdorffdim}.

\subsection{Hausdorff dimension of $\La$}\label{Hausdorffdim}
In Theorem \ref{tet1} we decomposed $\La$ into the diagonal of the square and countably many homothetic copies of $\La_{12}$ and $\La_{21}$, both attractors of self-similar IFS-s. Hence the Hausdorff dimension is the maximum of the dimension of the diagonal and $\dim_{\rm H} \La_{12}=\dim_{\rm H}\La_{21}$.
Note that the self-similar IFS $\mathcal F_{12}$ consists of $|E|$ similarities of contraction ratio $\frac1N$, and satisfies the Open Set Condition. As an immediate application of  \cite[Theorem 2.7]{Falconer}, the Hausdorff dimension of $\La_{12}$ and $\La_{21}$ is
$ \dim_{\rm H}\La_{12}= \frac{\log|E|}{\log N}$.

By this argument above we have proved the following theorem:
\begin{tm} %\label{tet2}
The Hausdorff dimension of $\La$ is
\[ \dim_{\rm H}\La= \max \left\{ \frac{\log|E|}{\log N}, 1\right\},
\]
furthermore,
\[ %\label{hdimnodiag}
\dim_{\rm H}\big(\La\!\setminus\!\mathrm{ Diag }\big) = \frac{\log|E|}{\log N}.
\]
\end{tm}
\begin{cor}
If $|V_1|=n_1=1$, then (\ref{degassum}) holds with $d_1=|E|$ in the bipartite $G$. Hence the degree distribution exponent (\ref{expon}) equals
\[\tilde\gamma= \frac{\log N}{\log |E|}=  \frac{1}{\dim_{\rm H}\big(\La\!\setminus\!\mathrm{ Diag }\big)}. \]
\end{cor}
\subsection{Average shortest path in $G_n$}%\label{ASHPath}

In many real networks, the typical distance between two randomly chosen points is of order $\log(|G |)$, the logarithm of the size of the network. We will see that our model also shares this property as well as the power law decay and the hierarchical structure, combining all these important features.

In this section we calculate the average length of shortest path between two nodes in $G_n$.
First we give a deterministic way to construct one of the shortest paths between any two nodes in the graph. To do so, we need to introduce some notation. Recall that the graph $G$ is a bipartite graph with partition $V_1$, $V_2$, see the beginning of Section \ref{detmodel}.
We remind the reader that for $\un x, \un y \in \Sigma_n$,  $\text{typ}(\un x)$, the common prefix $\un x\wedge \un y$ and the postfixes $\tilde {\un x}, \tilde{\un y} $ were defined in Definition \ref{typ}.

\begin{df}\label{blockdef} \

For two arbitrary vertices $\un x,\un y \in \Sigma_n$ we denote the length of their common prefix by  $k=k(\un x,\un y):= |\un x\wedge \un y|$.
Furthermore, let us decompose the postfixes $\tilde{ \un x}, \tilde{\un y}$  into blocks of digits of the same type:
\be \label{block}
\tilde {\un x }= \un b_1 \un b_2\dots \un b_r, \ \tilde {\un y}= \un c_1 \un c_2\dots \un c_q,
\ee
such that all of the blocks have a nonzero type and the consecutive blocks are of different types. That is, for $i=1, \dots, r-1, \ j= 1, \dots q-1$ we have
\[ \text{typ}(b_i)\neq \text{typ}(b_{i+1})\in \{1,2\},  \mbox{ and } \text{typ}(c_j)\neq \text{typ}(c_{j+1})\in \{1,2\}.\]
Note, that we denoted the number of blocks in $\tilde{\un x},\tilde{\un y}$ by $r$ and $q$, respectively.
 If $\un X$ and $\un Y$ are two random vertices of $G_n$, then the same notation as in (\ref{block}) is used with capital letters.
\end{df}
Now we fix an arbitrary self-map $p$ of $\Sigma_1$ such that
\[(x, p(x) ) \in E(G) \ \forall x \in G. \]
Most commonly, $p(p(x))\ne x$.
Note that $x$ and $p(x)$ have different types since $G$ is bipartite. For a word $\un z=(z_1 \dots z_m)$ with $\text{typ}(\un z )\in \{1,2\}$ we define $p(\un z):=(p(z_1) \dots p(z_m)).$. Then,
\be \label{pairedge}
(\un t\un z, \un t p(\un z)) \text{ is an edge in } G_{\ell+m}, \forall \un t=(t_1\dots t_\ell),
\ee
follows  from (\ref{edgerule}).

As usual we write $\text{Diam}(G)$ for the maximal graph-distance in the graph $G$ within components of $G$. Clearly $\text{Diam}(G)\le N-1$.
\begin{lm}\label{shortpathlemma}
Let $\un x, \un y$ be arbitrary vertices in the same connected component of $G_n$. Using the notation above,  the length of the shortest path between them is at least $r+q-1$ and at most $r+q+\text{Diam}(G)-2$.
\end{lm}
Considering the worst case scenario, i.e. choosing all blocks of length $1$ yields:
\begin{cor}\label{diam}
The diameter of the graph $G_n$ is at most $2n+\text{Diam}(G)-2$. Since the size of the graph is $N^n$, therefore \[\text{Diam}(G_n)=\frac{2}{\log{N}} \log(|G_n|)+O(1).  \]
\end{cor}

\begin{proof}[Proof of Lemma \ref{shortpathlemma}]
First we construct a path $P(\un x, \un y)$ of  minimal length.
Starting from $\un x$ the first half of the path $P(\un x, \un y)$ is as follows:
\[
\ba %\label{xhat}
\hat {\un x}^0&= \un x = (\un x \wedge \un y) \un b_1 \dots \un b_{r-1} \un b_r\\
\hat {\un x}^1&= (\un x \wedge \un y) \un b_1 \dots \un b_{r-1}p(\un b_{r})\\
&\dots \\
\hat {\un x}^{r-1}&= (\un x \wedge \un y) \un b_1 p(\un b_2 \dots p(\un b_{r-1}p(\un b_{r}))),\\
\ea
\]
Starting from $\un y$ the first half of the path $P(\un x, \un y)$ is as follows:
\[
\ba \label{yhat}
\hat {\un y}^0&= \un y = (\un x \wedge \un y) \un c_1 \un c_2\dots  \un c_r\\
\hat {\un y}^1&= (\un x \wedge \un y) \un c_1 \dots \un c_{r-1}p(\un c_{r})\\
&\dots \\
\hat {\un y}^{q-1}&= (\un x \wedge \un y) \un c_1 p(\un c_2 \dots p(\un c_{r-1}p(\un c_{q}))).\\
\ea
\]
It follows from (\ref{pairedge}) that
\[\ba
P_x:&=(\hat {\un x}^0, \hat {\un x}^1, \dots,\hat {\un x}^{r-1})\\
P_y:&=(\hat {\un y}^{q-1}, \dotsm \hat {\un y}^1, \hat {\un y}^0)
\ea \]
are two paths in $G_n$.
To construct $P(\un x, \un y)$ the only thing remained is to connect $\hat {\un x}^{r-1}$ and $\hat {\un y}^{q-1}$. Using (\ref{pairedge}) it is easy to see that this can be done with a path $P_c$  of length at most $\text{Diam}(G)$. In this way,
\[
P(\un x, \un y ):= P_x P_c P_y. \]
Clearly,
\[ %\label{lengthP}
r+q-1\le\mathrm{Length}(P(\un x, \un y ))\le r+q+\text{Diam}(G)-2
\]
On the other hand, now we prove that no shorter paths exists than $P(\un x, \un y)$.
Recall that it follows from (\ref{edgerule}) that for any path $Q(\un x,\un y)= (\un x = \un q^0, \dots, \un q^\ell = \un y)$, the consecutive elements of the path only differ in their postfixes, which have different types. That is,
\[ \forall i, \un q^i= \un w^i \un z^i, \ \un q^{i+1}= \un w^i \tilde{\un z}^{i}, \text{ with } \text{typ}(\un z^i)\ne \text{typ}(\tilde{ \un z}^{i})\in \{1,2\}.
\]
This implies that in each step on the path, the number of blocks in (\ref{block}) changes by at most one.
Recall that $|\un x \wedge \un y|=k$, so $x_{k+1}\ne y_{k+1}$.
Since the digit on the $k+1$-th position changes on the path,  we have to reach a point where all the digits to the right from the $k$-th position are of the same type. Starting from $\tilde{\un p}^0=\un x$, to reach the first vertex $\un a$ of this property, we need at least $r-1$ steps on any path $\tilde P$, where $r$ was defined in formula (\ref{block}). Similarly, starting from $\un y$, we need at least $q-1$ steps to reach the first vertex $\un b $ where all the digits after the $k$-th position are of the same type. Because $x_{k+1}\ne y_{k+1}$, we need at least one more edge and at most $\text{Diam}(G)$ edges.  \end{proof}

 \begin{tm}\label{averst}
The expectation of the length of a shortest path between two uniformly chosen vertices $\un X, \un Y \in G_n$ can be bounded by
\[ \frac{4 n_1 n_2}{N^2} (n-1)<\Ev(|P(\un X, \un Y)|) <N+\frac{4 n_1 n_2}{N^2} (n-1). \]
\end{tm}
\begin{cor}
The magnitude of the average length of a shortest path between two uniformly chosen vertices in $G_n$ is the logarithm of the size of $G_n$, which is the same order as $\text{Diam}(G_n)$.
\end{cor}
\begin{proof}[Proof of Theorem \ref{averst}]
Let $\un X, \un Y$ be independent, uniformly chosen vertices of $G_n$. In this proof we use the notation introduced in Definitions \ref{typ} and \ref{blockdef}.  The digits of the code of a uniformly chosen vertex are independent and uniform in $\{ 0, \dots, N-1\}$, hence $K(\un X, \un Y):=|\un X \wedge \un Y|$ has a truncated geometric distribution with parameter $\frac{N-1}{N}$. That is
\[ \Pv(K(\un X, \un Y)=k)= \left\{ \ba &\left( \frac{1}{N}\right)^{k}\cdot \frac{N-1}{N}, \mbox{ if } 0\le k<n,\\&\left( \frac{1}{N}\right)^{n} \mbox{ if } k=n. \ea \right.  \]

Furthermore, given that the length of the prefix is $k=K(\un X, \un Y)$, the random variables $R$ and $Q$ (see Definition \ref{typ}) can be represented as the sum of indicators corresponding to the start of a new block:
\[ \ba R&= 1+\sum\limits_{i=1}^{n-k-1}\ind_{\text{typ}(X_{k+i})\ne \text{typ}(X_{k+i+1})},\\
 Q&= 1+\sum\limits_{i=1}^{n-k-1}\ind_{\text{typ}(Y_{k+i})\ne \text{typ}(Y_{k+i+1})}. \ea\]
%We introduced indicators counting the start of a possible new block at each coordinate.
%\[  R= 1+\sum\limits_{i=1}^{n-k-1}\ind_{\text{typ}(X_i)\ne \text{typ}(X_{i+1})},\
 %Q= 1+\sum\limits_{i=1}^{n-k-1}\ind_{\text{typ}(Y_i)\ne \text{typ}(Y_{i+1})}. \]
 Taking expectation yields
\[
\begin{split}
\Ev(Q|K(\un X, \un Y)=k)&=\Ev(R|K(\un X, \un Y)=k)\\
&=1+\Ev\left(\sum\limits_{i=1}^{n-k-1}\ind_{\text{typ}(
X_{k+i})\ne \text{typ}(X_{k+i+1})}\right)\\
&=1+\sum\limits_{i=1}^{n-k-1}\Pv(\text{typ}(X_{k+i})\ne \text{typ}(X_{k+i+1}))\\
&= 1+(n-k-1) \frac{2 n_1 n_2}{N^2}.
\end{split}
\]

So weighting this with the geometric weights of the length of the prefix, we get
\[
\begin{split}
\Ev(Q)=\Ev(R)&=\Ev\big( \Ev(R|K(\un X, \un Y))\big)\\
&=\Ev\big( 1+(n-K(\un X, \un Y)-1) \frac{2 n_1 n_2}{N^2}\big)\\
&= 1+ \left(n-\frac{1}{N-1}\left(1-\frac{1}{N^n}\right)-1\right) \frac{2 n_1 n_2}{N^2}.
\end{split}
\]
Using this and the following immediate consequence of  Lemma \ref{shortpathlemma}
\[%\label{9}
-1\le\Ev(|P(\un X,\un Y)|-(R+Q))\le \text{Diam}(G)-2, \]
finally we obtain that
\[
1-\frac{1}{N-1}+ \frac{4 n_1 n_2}{N^2} (n-1)\le \Ev(|P(\un X,\un Y)| < \text{Diam}(G) +\frac{4 n_1 n_2}{N^2} (n-1).\]
\end{proof}

\subsection{Decay of local clustering coefficient of the modified sequence $\Big\{\hat G_n\Big\}$ }\label{clust}
An important property of most real networks is the high degree of clustering.
In general, the local clustering coefficient of a node $v$ having $n_v$ neighbors is defined as
$$
C_v:=\frac{\#\{ \mbox{links between neighbors of }v\}}{{n_v \choose 2}}.
$$
Note that the numerator in the formula is the number of triangles containing $v$ and $C_v$ is the portion of the pairs of neighbors of $v$ which form a triangle with $v$ in the graph.

Observe that without the loops the graph sequence $G_n$ is bipartite, i.e. there are no triangles in the graph $G_n$. However, we can modify the graph sequence $G_n$ in a natural  way, like in \cite{AB},  to get a new sequence $\hat G_n$ preserving the hierarchical structure of $G_n$, still reflecting the dependence of clustering coefficient on node degree observed in several real networks. Namely, the local clustering coefficient of a vertex $v$ is of order $1/\deg(v)$.

\begin{figure}
\centering
\subfigure[We obtain $\widehat{G}$  by adding the dashed (red) edges to $G$.]{\label{17}
\includegraphics[keepaspectratio,height=3cm]{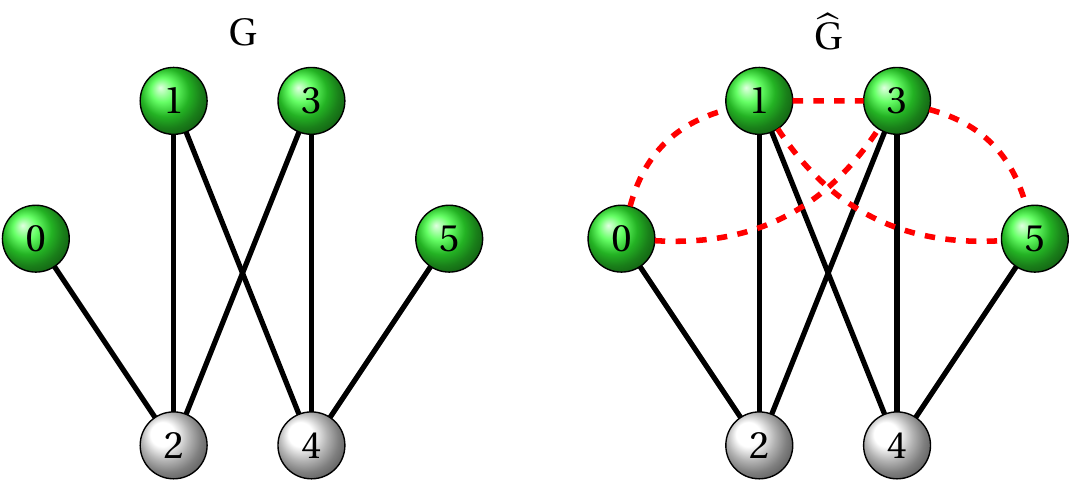}}

\subfigure[$\widehat{G}_2$: The edges of $\widehat G_2$ and $G_2$  differ only at the lowest hierarchical level (cf. Figure \ref{19})]{\label{18}
\includegraphics[keepaspectratio,width=12cm]{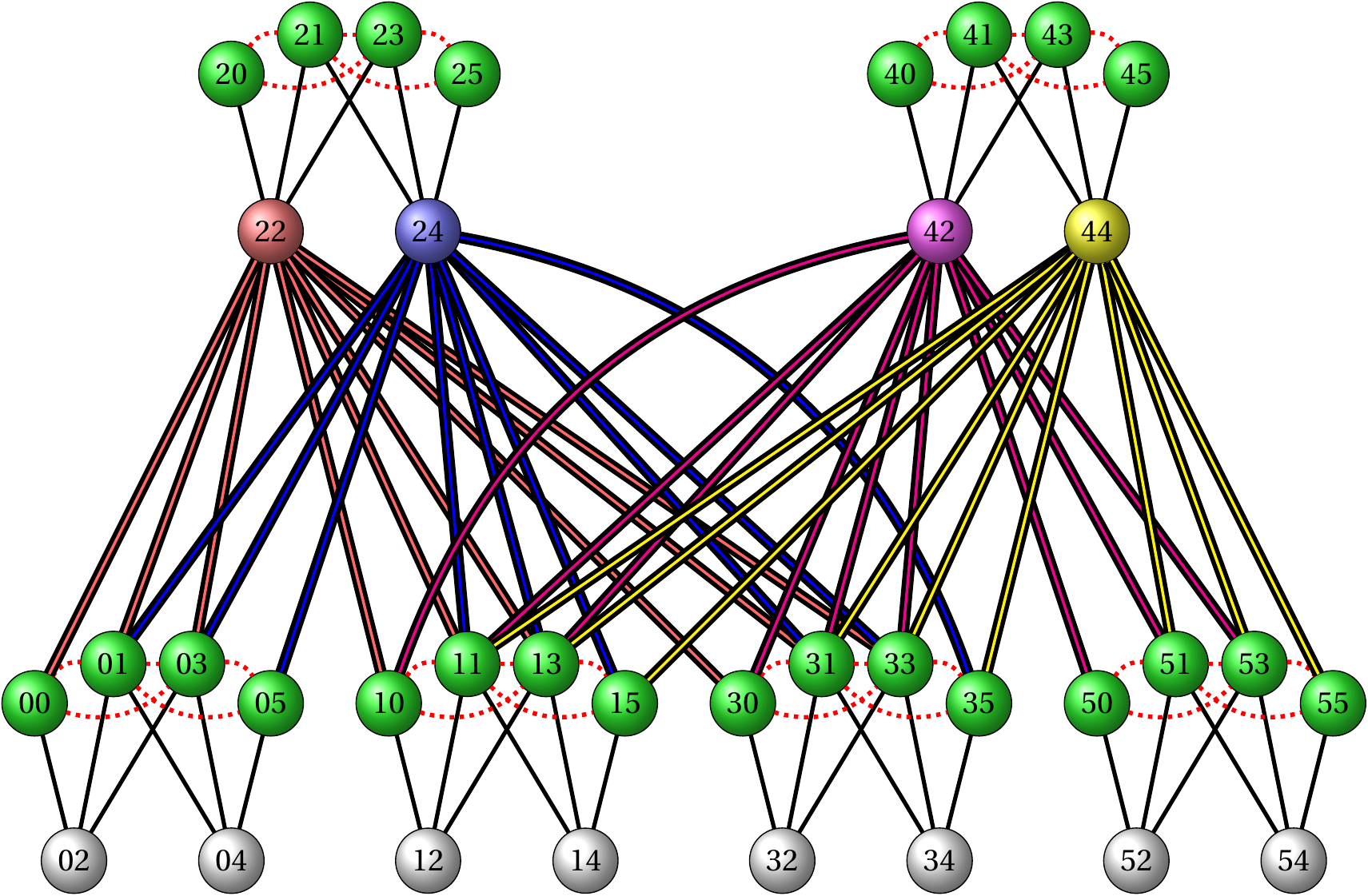}}
\caption{Clustering extended "fan". }\label{20}

\end{figure}

\begin{df}%\label{gkalap}
\begin{itemize}\
  \item We obtain the graph $\hat G$  adding a set of extra edges $RE(\hat G)$ to $G$ satisfying the following property:

    \textbf{Property R}

      $\forall x\in\Sigma _1,\ \exists y,z\in \Sigma _1$,
such that  two among the edges of the triangle $(x,y,z)_\Delta $ are contained in $E(G)$ and one of the edges is in $RE(\hat G) $.

\medskip
 So,
 $$
V(\hat G )=V(G) \mbox{ and } E(\hat G)=E(G)\cup RE(\hat G).
 $$
\noindent In the example presented on  Figure \ref{20}  the edges from $RE(\hat G)$ are the dashed red edges.
  \item Similarly we define the graph sequence $\left\{\hat G_n\right\}_{n=1}^{\infty }$ by deleting all loops in $G_n$ and adding extra edges to $G_n$. That is, the vertices $V(\hat G_n)=V(G_n)=\Sigma_n$, and with the definition of the simple graph $G'_n$ in Section \ref{selfloop}, the edge set is extended by the following rule
\begin{equation}\label{HEGN}
   E(\hat G_n)= E(G'_n)\bigcup RE(\hat G_n),
\end{equation}
    where
    \be\label{edgerulehat}
    RE(\hat G_n)= \left\{ {x_1 \dots x_n \choose y_1 \dots y_n}: x_i=y_i, i\le n-1, {x_n\choose y_n} \in RE(\hat G)\right\}.
    \ee
\end{itemize}
\end{df}
It  is clear from Property {\bf R} that
\be\label{cmin}
\hat C_{\min}:=\min\limits_{x\in \hat G} C_x>0.
\ee
Further, using (\ref{edgerule}) and  (\ref{edgerulehat}) one can easily see that
the degree of a vertex $\un x \in \hat G_n$ is
\be\label{deghat}
\widehat \deg_n(\un x)=S(\un x)\cdot \deg(x_n) + \left(\widehat \deg(x_n)-\deg(x_n)\right),
\ee
where $\widehat \deg(.)$ denotes the degree of a vertex in $\hat G$, while $\deg(.)$ stands for the degree in $G$.
\begin{rem}
The difference between the degree of any node $\un x \in \Sigma_n$ in $G_n$ and in $\hat G_n$ is bounded, thus the degree sequence of $\hat G_n$ has the same power law exponent as $G_n$.
\end{rem}
\begin{tm}\label{clustertetel}
There exists $K_1, K_2> 0$ such that the local clustering coefficient $C_{\un x}$ of an arbitrary node $\un x \in \hat G_n$ satisfies.
\[%\label{clustertet}
 \frac{K_1}{\widehat{\deg}_n(\un x)}\le C_{\un x}\le \frac{K_2}{\widehat{\deg}_n(\un x)}.
\]

\end{tm}
\begin{proof}
We write  $\mathcal{T}_n(\un x)$ for the set of all triangles in $\hat G_n$ containing the node $\un x\in \Sigma _n$. We say that
 a triangle $(\un x, \un y, \un z)_\Delta\in \mathcal{T}_n(\un x)$ is regular if and only if exactly two of its edges are from $E(G_n)$. The triangle $(\un x, \un y, \un z)_\Delta  \in \mathcal{T}_n(\un x)$ is called irregular if it is not regular. The set of irregular triangles containing $\un x$ is denoted by $\mathcal{IRT}_n(\un x)$. We partition
the set of regular triangles $\mathcal{RT}_n(\un x)$ into the classes:
$$
\mathcal{RT}_n(\un x)=\mathcal{RT}_n^1(\un x)\cup \mathcal{RT}_n^2(\un x)
$$
in the following way: A triangle $(\un x,\un y,\un z)_\Delta\in \mathcal{RT}_n(\un x)$ belongs to $\mathcal{RT}_n^1(\un x)$ if and only if $\un x$ is NOT an endpoint of the edge contained in $RE(\hat G_n)$. That is
$$
\mathcal{RT}_n^1(\un x):=
\left\{(\un x,\un y,\un z)_\Delta \in \mathcal{RT}_n(\un x):
{ \un x \choose \un y },{ \un x \choose \un z }\in E(G_n).
\right\}
$$
Hence, $\mathcal{RT}_{n}^{2}(\un x)$ is the set of those $(\un x,\un y,\un z)_\Delta \in \mathcal{RT}_n(\un x)$ for which either ${ \un x \choose \un y }\in E(G _n)$ and ${ \un x \choose \un z }\in RE(\hat G _n)$ or vice versa.
Summarizing these partitions:
\[%\label{10}
\mathcal{T}_n(\un x)=\mathcal{RT}_n(\un x)\cup \mathcal{IRT}_n(\un x)=
\mathcal{RT}_n^1(\un x)\cup \mathcal{RT}_n^2(\un x)\cup\mathcal{IRT}_n(\un x)
\]
Now we define the cardinality of these classes:
$$
\Delta_n^1 (\un x):=\#\mathcal{RT}^1(\un x), \Delta_n^2 (\un x):=\#\mathcal{RT}^2(\un x)\mbox{ and } \Delta_n^{\mathrm{ir}} (\un x):=\#\mathcal{IRT}(\un x).
$$
When $n=1$ then we suppress the index $n$.   Observe that by Property {\bf R},
$$\Delta_n^\mathrm{r}(\un x):=\Delta_n^{1}(\un x)+\Delta_n^{2}(\un x)\geq 1, \quad \forall n\geq 1, \un x\in \Sigma _n.$$
Now we compute $\Delta _{n}^{i}(\un x)$, $i\in \{1,2, \mathrm{ir}\}$, for an arbitrary fixed $\un x\in \Sigma _n$. To do so the notation $\ell(\un x)$ will be used.
First we verify that
\begin{equation}\label{szamtyp1}
  \Delta _{n}^{1}(\un x)=
  \sum\limits_{r=0}^{\ell(\un x)-1}\prod_{j=1}^{r}\deg(x_{n-j})\cdot \Delta ^1(x_n)=
 S(\un x)\cdot \Delta ^1(x_n),
\end{equation}where $S(\un x)$ was defined in (\ref{6}).
To see this, observe that it follows from
(\ref{edgerule}), (\ref{HEGN}) and (\ref{edgerulehat}) that
\[%\label{1}
  (\un x,\un y,\un z)_\Delta \in \mathcal{RT}_{n}^{1}(\un x)
\]
holds if and only if all of the following three assertions are satisfied:
\noindent\begin{enumerate}
  \item $
 \exists 0\leq r\leq \ell(\un x)-1,\
 |\un y\wedge \un z|=n-1 \mbox{ and }
 |\un x\wedge \un y|=|\un x\wedge \un z|=n-r-1$
\item $ { x_k \choose y_k }\in E(G)\mbox{ whenever }n-r\leq k\leq n-1$
\item $  (x_n,y_n,z_n)_\Delta \in \mathcal{RT}^1(x_n).$
\end{enumerate}
Hence (\ref{szamtyp1}) is obtained by an immediate calculation.

Now we prove that
\begin{equation}\label{szamtyp2}
  \Delta _{n}^{2}(\un x)=
  \sum\limits_{r=0}^{\ell(\un x)-1}\prod_{j=1}^{r}\deg(x_{n-j})\cdot \Delta ^2(x_n)=
S(\un x)\cdot \Delta ^2(x_n).
\end{equation}
This is so because  by (\ref{edgerule}),  (\ref{HEGN}) and (\ref{edgerulehat}) we have
\[%\label{1a}
  (\un x,\un y,\un z)_\Delta \in \mathcal{RT}_{n}^{2}(\un x)
\]
holds if and only if all of the following three assertions are satisfied:
\begin{enumerate}
\item $ \exists 0\leq r\leq \ell(\un x)-1,\
 |\un x\wedge \un y|=n-1 \mbox{ and }
 |\un x\wedge \un z|=|\un y\wedge \un z|=n-r-1,$
\item $ { x_k \choose z_k }\in E(G)\mbox{ whenever }n-r\leq k\leq n-1$
\item $  (x_n,y_n,z_n)_\Delta \in \mathcal{RT}^2(x_n).$
\end{enumerate}
Hence, using the same argument as above we get (\ref{szamtyp2}).

Finally, we determine the number of irregular triangles containing $\un x$:
\begin{equation}\label{szamtyp3}
\Delta _{n}^{\mathrm{ir}}(\un x)=\Delta^{\mathrm{ir}} (x_n).
\end{equation}
This follows from the fact that
$$
 (\un x,\un y,\un z)_\Delta \in \mathcal{IRT}_{n}(\un x)
$$
is equivalent to
$$
\forall 1\leq i\leq n-1,\ x_i=y_i=z_i\mbox{ and } (x_n,y_n,z_n)_\Delta \in \mathcal{IRT}(x_n).
$$
We write $Z_\Delta (\un x)$ for the number of all triangles in $\hat G _n$ containing $\un x$:
$$
Z_\Delta (\un x):=\underbrace{\Delta _{n}^{1}(\un x)+\Delta _{n}^{2}(\un x)}_{\Delta ^{\mathrm{r}}(\un x)}+\Delta _{n}^{\mathrm{ir}}(\un x).
$$
Using (\ref{deghat}), (\ref{szamtyp1}), (\ref{szamtyp2}) and (\ref{szamtyp3}) we get
\be\label{zdelta}
C_{\un x}= \frac{Z_\Delta(\un x)}{{ \widehat{\deg_n} (\un x) \choose 2}} = \frac{2 \Delta^{\mathrm{r}}(x_n)\cdot
S(\un x)+ 2\Delta^{\mathrm{ir}}(x_n)}{\widehat{\deg_n}(\un x)(\widehat\deg_n(\un x)-1)},
\ee
where $S(\un x)$ was defined in (\ref{6}).
Now we estimate $C_{\un x}$.
\begin{cl}\
\begin{description}
  \item[(i)] If $\ell(\un x)=1$, then $C_{\un x}= C_{x_n}$.
  \item[(ii)] If $\ell(\un x )\ge 2$, then we have
\begin{equation}\label{estC}
  \left|C_{\un x} - \frac{2 \Delta^{\mathrm{r}}(x_n)}{\deg (x_n)}\cdot \frac{1}{\widehat \deg_n(\un x)}\right|\le \frac{\mathrm{const}}{\widehat\deg_n^2 (\un x)}.
\end{equation}
\end{description}
\end{cl}
\begin{proof}[Proof of the Claim]
Part \textbf{(i)} immediately follows from (\ref{edgerule}). To prove \textbf{(ii)} we fix an arbitrary $\un x\in \Sigma _n$ with $\ell(\un x)\geq 2$. Since $t,u,v$ introduced below depend only on $x_n$ there exists a constant $C_*$ independent of $n$ and $\un x$ such that
\begin{equation}\label{bound}
  0\leq t:=\frac{\Delta ^r(\un x)}{\deg(x_n)},\quad u:=\widehat{\deg}( x_n)-\deg(x_n),\quad v:=2\Delta ^{\mathrm{ir}}(x_n)<C_*.
\end{equation}
To prove (\ref{estC}) it is enough to verify that
\[%\label{11}
Q:=\left(\widehat{\deg_n}(\un x)\right)\left(\widehat{\deg_n}(\un x)-1\right)\cdot C_{\un x}-2t\cdot (\widehat{\deg_n}(\un x)-1)
\]
is bounded in $n$ and $\un x\in \Sigma _n$.
This so, because by  (\ref{deghat}) and (\ref{zdelta}) we have
\begin{eqnarray*}
% \nonumber to remove numbering (before each equation)
Q &=& 2\Delta ^{r}(x_n)\cdot S+v-2t\big(\underbrace{S\cdot \deg(x_n)+u}_{\widehat{\deg}(\un x)}-1 \big)\\
 &=& 2\Delta ^{r}(x_n)\cdot S+v- \underbrace{2\Delta ^{r}(x_n)\cdot S}_{2tS\deg(x_n)}-2t(u-1)\\
 &=&  v-  2t(u-1), \end{eqnarray*}
which is bounded by (\ref{bound}). \end{proof}
Property {\bf R} implies that both $C_{x_n}$ and $\frac{\Delta^{\mathrm{r}}(x_n)}{\deg(x_n)}$ are bounded away from zero. This completes the proof of the Theorem \ref{clustertetel}.
\end{proof}

The following theorem shows that the graph sequence $\hat G_n$ displays similar features to that of considered in \cite{AB}, namely, the average local clustering coefficient of the graphs $\hat G_n$ is not tending to zero with the size of $\hat G_n$.

\begin{tm}\label{clusteraver}
The average local clustering coefficient $\bar C(\hat G_n)$ of the graph $\hat G_n$ is bounded by two positive constants, more precisely
\be\label{claver}
\frac{2 n_1 n_2\hat C_{\min}}{N^2}\le \bar C(\hat G_n)\le \bar C(\hat G),
\ee
where $\hat C_{\min}$ was defined in (\ref{cmin}).
\end{tm}
\begin{proof}

We will use the notation introduced in the proof of Theorem \ref{clustertetel}.
It easily follows from the proof of Theorem \ref{clustertetel} that
\be\label{clusterxn} C_{\un x} \le C_{x_n}. \ee
Namely, if
          $\ell(\un x)=1$ then by (\ref{edgerulehat}), $C_{\un x}=C_{x_n}$.
           If $\ell(\un x)\geq 2$ then $S(\un x)\geq 1$ thus using (\ref{zdelta}) we obtain
   \begin{eqnarray*}%\label{Cfelso}
   %  to remove numbering (before each equation)
     C_{\un x} &\leq &\underbrace{\frac{\Delta ^{\mathrm{r}}(x_n)+\Delta ^{\mathrm{ir}}( x_n)}{{\deg(x_n) \choose  2}}}_{C_{x_n}}
     \cdot
     \underbrace{{\deg(x_n) \choose  2}
     \cdot
     \frac{S(\un x)}{{\widehat{\deg_n}(\un x) \choose 2}}}_{\leq 1}.
       \\
\nonumber      &\leq & C_{x_n}.
   \end{eqnarray*}
This completes the proof of (\ref{clusterxn})  from which the upper estimate of (\ref{claver}) follows by averaging.
 On the  other  hand to see that the lower estimate holds we take into consideration only the contribution of $\un x \in \Sigma_n$ with $\ell(\un x )=1.$
$$
  \bar C(\hat G_n)> \frac{1}{N^n} \left(\sum_{z\in V_1}   N^{n-2} n_{ 2} C_z +\sum_{z\in V_2}   N^{n-2} n_{ 1} C_z\right)
  $$
Using $C_z> \hat C_{\min}$, the lower bound of (\ref{claver}) follows.
\end{proof}

\section{Definition of the randomized model}\label{rand}
In this section we randomize the deterministic model in Section \ref{detmodel} by using  $\La$ in $[0,1]^2$.
%%%%%%%%%%%%%%%
The random graph sequence $G_n^{\text{r}}$ is generated in a way which was
inspired by the $W$-random graphs introduced  by Lov\'asz and  Szegedy \cite{LovSzeg}. See also \cite{BJR}.

%%%%%%%%%%%%%%%%%%

Fix a deterministic model with a base graph $G$, $|V(G)|=N$. This determines $\Lambda(a,b)$ the limit of the sequence of scaled adjacency matrices, see the definition (\ref{Lambda})  and (\ref{Lambdan}) in Section \ref{sectionadjac}. Now for each $n$, we throw $M_n+1$ independent, uniform random numbers over $[0,1]$:
\[ X^{(1)},X^{(2)},\dots,X^{(M_n+1)} \sim U[0,1], \mbox{ i.i.d. } \]
We denote the $N$-adic expansion of each of these numbers by
\[ \un X^{(i)}=(X_1^{i},X^{i}_2,\dots), \quad\mbox{i.e. } X^{(i)}=\sum_{k=1}^{\infty}\frac{X^i_k}{N^k},\]
where the $X_k^i$-s are uniform over the set $\{0,1,\dots, N-1\}$. The n-th approximation of $X^{(i)}$ is
\[  X^{(i)}_{[n]}  =\sum_{k=1}^{n}\frac{X^i_k}{N^k}, \quad \un X^{(i)}_{n} = ( X_1^{i}, \dots, X^{i}_n) .\]

Now we construct the random graph $G^{\mathrm{r}}_n$ as follows: $|V(G^{\mathrm{r}}_n)|=\{1, \dots, M_n\}$, and $E(G^{\mathrm{r}}_n)$ is given by
\[E(G^{\mathrm{r}}_n)=\left\{ (i,j)\big|\  \text{int} \left( I_{X^{(i)}_{[n]}}\times I_{X^{(j)}_{[n]}} \right)\cap \Lambda \ne \emptyset \right\},\]
where $\text{int}$ denotes the interior of a set. Clearly,
\[E(G^{\mathrm{r}}_n)=\left\{ (i,j)| \La_n(X^{(i)},X^{(j)})=1 \right\}.\]
Note that
\[ \La_n(X^{(i)},X^{(j)})=1 \ \Leftrightarrow \ {X^i_1\dots X^i_n \choose X^j_1\dots X^j_n} \in E(G_n).    \]
Namely, we can think of the first $n$ digits $(X^i_1,\dots, X^i_n)$ and $(X^j_1,\dots, X^j_n) $ of the N-adic expansion of $X^{(i)}$ and $ X^{(j)}$ as vertices in $G_n$.  We draw an edge between the two vertices $i$ and $j$ in $G_n^\text{r}$ if  the vertices  $(X^i_1\dots X^i_n)$ and $(X^j_1\dots X^j_n)$ are connected by an edge in the deterministic model $G_n$. This gives the following probabilistic interpretation of the random model:
\begin{rem}\label{prob}
Consider the deterministic graph sequence $G_n$ with urns sitting at each vertex $v \in G_n$. Now throw $M_n+1$ balls independently and uniformly into the urns, and connect vertex $i$ to vertex $j$ by an edge in the random graph $G_n^\mathrm{r}$ if and only if the urns of ball $i$ and $j$ are connected by and edge in $G_n$.
\end{rem}

We need to introduce some further notation.

{\bf Frequently used definitions. }
Under assumption \ref{degassum},  for an $\un x\in G_n$ with $\ell(\un x)=k$ the degree of $\un x$ is
\[ t_k := \frac{d_1^{k+1}-1}{d_1-1}+1, \]
independently of the length of $n$.

In the random graph $G_n^{\text{r}}$, the conditional probability of the degree distribution of a random node $V\in \{ 0, \dots, M_n \}$ conditioned on the first $n$ digits of the N-adic expansion of the corresponding code $X^{(V)}$ follows a Binomial distribution:
\be \label{bin}\big(\deg(V)|(X^{V}_1 \dots X^V_n)=\un x  \big)\sim BIN \left(M_n, \frac{t_{\ell(\un x)}}{N^n}\right).\ee
This follows from the characterization of $G_{n}^{\mathrm{r}}$ described  in Remark \ref{prob}. Namely, assume  that the  $V$-th ball has landed in urn with label $\un x \in \Sigma_n$. In $G_n$ there are exactly $\deg_n(\un x)-1=t_{\ell(\un x)}$ vertices $\un y\in \Sigma _n$ that are connected to $\un x$. All the balls landing into urns corresponding to these vertices $\un y$ will be connected to $V$ in $G_n^{\text{r}}.$
\section{Properties of the randomized model}\label{proprand}
In this section we determine the proportion of isolated vertices and characterize the degree sequence.
\subsection{Isolated vertices}
\begin{tm}\label{izo}
If $M_n=c_n N^n$ with $\lim\limits_{n\to \infty}c_n=\infty$, then the fraction of isolated vertices tends to zero as $n\to \infty$. More precisely, for a uniformly chosen node $V \in G_n^\text{r}$,
\[ \Pv\left(\deg(V)=0\right)\le e^{- d_{\min} c_n},\]
where $d_{\min}$ stands for the minimal degree in the base graph $G$, and in $\deg(.)$ we do not count the loops.
\end{tm}
The following corollary is an immediate consequence of the Borel-Cantelli lemma.
\begin{cor}%\label{borelc}
If $\sum\limits_{n=1}^{\infty} c_n N^n e^{-d_{\min}c_n}<\infty$, then almost surely there will be only finitely many $n$-s, for which the graph $G^\mathrm{r}_n$ has isolated vertices.
\end{cor}
The assumption of the Corollary is satisfied if e.g. $c_n >n\log (N +1)$.
\begin{proof}[Proof of Theorem \ref{izo}]
Given the N-adic expansion of $X^{(V)}$, the probability that a vertex is isolated depends on how many neighbors the vertex $(X^V_1\dots X^V_n)$ has in the deterministic model.
So we can write
\[ \Pv\left(\deg(V)=0\right)= \sum_{\un x\in \Sigma_n} \Pv(\deg(V)=0| (X^V_1\dots X^V_n)=\un x )\cdot \frac{1}{N^n} \]
As we have already seen, $\big(\deg(V)| (X^V_1\dots X^V_n)=\un x\big)$ follows a Binomial distribution with parameters $M_n$ and $\frac{\deg_n (\un x)-1}{N^n}$, so the conditional probability of  isolation is
\[ \ba \Pv(\deg(V)=0| (X^V_1\dots X^V_n)=\un x )&= \left(1-\frac{t_{\ell(\un x)}}{N^n} \right)^{M_n}\\ &\le e^{-\deg_n (\un x) c_n }(1+o(1)). \ea\]
Obviously $e^{-\deg_n(\un x) c_n }\le e^{-d_{\min} c_n }$ holds for all $\un x\in \Sigma_n$, which completes  the proof.
\end{proof}
\subsection{Decay of degree distribution}
Fix a constant $K$ such that for a standard normal variable $Z$, $\Pv(|Z|>K)<e^{-10}.$
We write
\[ I_{k,n}:=[c_n t_k-K \sqrt{ c_n t_k} , c_n t_k+K \sqrt{ c_n t_k}], \]
and
\[ k_0(n):=\max\left\{ (n+1)\frac{\log d_2}{\log d_1}, \frac{\log n}{\log d_1}\right\}.\]
Now we describe the degree distribution for the random model.
\begin{tm}\label{randdeg}
Let $k>k_0(n)$ and $u\in I_{k,n}$. Then for a uniformly chosen node $V$ in $G_n^{\text{r}}$
\[  \Pv\left(\deg (V)= u\right) = \left(\frac{n_1}{N}\right)^{k} \frac{n_2}{N}
\cdot \frac{1}{\sqrt{c_n t_k} }\phi\Big(\frac{u-c_n t_k}{\sqrt{c_n t_k (1- \frac{t_k}{N^n})}}\Big)\big(1+O(\frac{1}{\sqrt{c_n t_k}})\big),\]
where $\phi$ denotes the density function of a standard Gaussian variable.
%The limit distribution at other points \[ \lim_{n\to \infty} \Pv\left(\deg (V)\not\in \{[t_k-K \sqrt{t_k}, t_k+K \sqrt{t_k}], k\in \mathbb N\}\right) \le e^{-10}.\]
\end{tm}
This immediately implies
\begin{cor}%\label{randdega}
The degree distribution of the random model is given by the following formula for $a,b\in [-K,K]$:
\[  \ba \Pv\left(\deg (V)\in [c_n t_k+a \sqrt{c_n t_k}, c_n t_k+b \sqrt{c_n t_k}]\right)& = \left(\frac{n_1}{N}\right)^k \frac{n_2}{N}\cdot\left(\Phi(b)-\Phi(a)\right)\\&+ O\Big( \Big(\frac{n_1}{N}\Big)^k\frac{1}{\sqrt{c_n t_k}}\Big),\ea \]
where $k>k_0(n)$ and $\Phi$ denotes the distribution function of a standard Gaussian variable.
So, for $u\in I_{k,n}, \ k> k_0(n)$ the tail of the probability distribution is:
\be\label{tail} \ba \Pv(\deg(V)>u)& =
\left( \frac{n_1}{N}\right)^{k+1}+ \left(\frac{n_1}{N}\right)^k \frac{n_2}{N}
\left( 1-\Phi\Big(\frac{u-c_n t_k}{\sqrt{c_n t_k (1-\frac{t_k}{N^n})}}\Big)\right) \\
&+\left(\frac{n_1}{N}\right)^{k+1} O\Big(\frac{1}{\sqrt{c_n t_k}}\Big).\ea \ee
\end{cor}

This holds because $\Pv(\deg(V)>u)$ equals the sum of all probability mass that is concentrated around $t_l$-s for $l\ge k+1$, resulting in the first term, plus the second term coming from the part greater than $u$ of the binomial mass around $t_k$.
As a consequence, the decay of the degree distribution follows a power law. Namely, the following holds
\begin{tm}\label{decay}Let
$$\gamma:=1+\frac{\log(\frac{N}{n_1})}{\log d_1}.$$
Then the decay of the degree distribution is:
\[  \Pv(\deg(V)>u)= u^{-\gamma +1}\cdot L(u),\]
where $L(u)$ is a bounded function:
\[ \frac{n_1}{N} \le L(u)\le \frac{N}{n_1}.\]
\end{tm}

\begin{proof}[The idea of the proof of Theorem \ref{randdeg}]
The conditional distribution of the degree of a node $V$ conditioned on the n-digit N-adic expansion of $X^{(V)}_n=\un x$ follows a $BIN(c_n N^n, \frac{t_{\ell(\un x)}}{N^n})$ law. This is close to a $POI(c_n t_{\ell(\un x)})$ random variable, because $c_n$ and $t_{\ell(\un x)}$  tend to infinity in a much smaller order than $N^n$.   Now for the $POI(c_n t_{\ell(\un x)})$ variable, the Central Limit Theorem holds with an error term of order $1/\sqrt{c_n t_{\ell(\un x)}}$. Now the unconditional degree distribution comes from the law of total probability and from the fact that all other errors are negligible.
\end{proof}

\begin{proof}[Proof of Theorem \ref{randdeg}]
We determined the degree distribution of the deterministic model under assumption (\ref{degassum}), see Section \ref{degdistr} for details. Recall that if $k> k_0(n)$, then the mass at $t_k$ is
\[ p_k:= \Pv(\ell(\un x)=k)=\left(\frac{n_1}{N}\right)^k \frac{n_2}{N}. \]
We show that in the random model $G_n^{\text{r}}$, these Dirac masses are turned into Gaussian  masses centered at $c_n t_k$. Suppose $u\in I_{k,n}$. By the law of total probability, we have
\be\label{12} \ba \Pv(\deg(V)=u) &= \Pv(\deg(V)=u  | \left(X_{1}^{V}\dots X_{n}^{V}\right)=\un x, \ell(\un x)=k) \cdot p_k\\
&+ S_1 +S_2 ,\ea\ee
where
\[  \ba S_1&= \sum_{j=1}^{k-1} \Pv(\deg(V)=u  | \left(X_{1}^{V}\dots X_{n}^{V}\right)=\un x, \ell(\un x)=j)\cdot p_j\\
S_2&=\sum_{j=k+1}^{n} \Pv(\deg(V)=u  | \left(X_{1}^{V}\dots X_{n}^{V}\right)=\un x, \ell(\un x)=j)\cdot p_j\ea \]
 $S_1$ and $S_2$ combines the total contribution of cases when  $\ell(X^V_1 \dots X^V_{n}) \neq k$, i.e. referring to the urn model of our random graph,  $S_1+S_2$ settles the cases when the random ball $V$ falls into an urn which has degree different from $t_k$ in $G_n$. As a first step in our proof we show that the right hand side in the first line of (\ref{12}) gives the formula in Theorem \ref{randdeg}, then as a second step we verify that $S_1+S_2$ is negligible.

{\bf First step:} Following the standard proof of the local form of de Moivre-Laplace CLT, we obtain that for $u \in I_{k,n}$
\[
  \begin{split}
 \mathbb{P}&\Big( \deg(V)=u  | \left(X_{1}^{V}\dots X_{n}^{V}\right)=\un x \Big)
\\
  &
= \frac{1}{\sqrt[]{c_nt_{\ell(\un x)}(1-\frac{t_{\ell(\un x)}}{N^n})}}
\phi\left(\frac{u-c_nt_{\ell(\un x)}}{\sqrt[]{c_nt_{\ell(\un x)}(1-\frac{t_{\ell(\un x)}}{N^n})}}\right)
\cdot  \left( 1+O\big(  \frac{1}{\sqrt{c_nt_{\ell(\un x)}}}\big)\right).
  \end{split}
\]
We can neglect  $1-\frac{t_{\ell(\un x)}}{N^n}$.  This completes the first step.

{\bf Second step:}  Since $u\in I_{k,n}$ we have:
\be\label{errors} \ba S_1 \le &\sum_{j=1}^{k-1}\Pv(\deg(V)> t_k-K\sqrt{t_k} | (X^V_1\dots X^V_n)=\un x, \ell(\un x)= j)\cdot p_j \\
S_2 \le  &\sum_{j=k+1}^{n}\Pv(\deg(V)< t_k+K \sqrt{t_k} | (X^V_1\dots X^V_n)=\un x, \ell(\un x)= j)\cdot p_j\ea \ee
Now we use the fact known from Chernoff-bounds: for an $Z \sim BIN(m,p)$ variable
\[ \Pv(Z \ge (1+\delta)\Ev(Z))\le e^{-\frac12 \delta^2 \Ev(Z)},  \] and the same bound holds for $\Pv(Z \le (1-\delta)\Ev(S))$.
By (\ref{bin}), to estimate each summand in (\ref{errors}) we can apply these inequalities for $Z_j \sim BIN(c_n N^n, \frac{t_{j}}{N^n}), \ j\in \{1,\dots, n\}\setminus \{ k\}$, yielding an upper bound
\[\ba S_1+S_2&\le \sum_{j=1}^{k-1}e^{-\frac12 d_1^{2k-j} c_n}\cdot p_j  +\sum_{j=k+1}^n e^{-\frac12 (1-d_1^{k-j})^2 d_1^j c_n}\cdot p_j \\&\le e^{-\frac18 d_1^{k} c_n}.\ea\]
Since $e^{-\frac18 d_1^{k} c_n}=o(\frac{1}{\sqrt{c_n t_k}})$, the statement of Theorem \ref{randdeg} follows.
\end{proof}
Now we are ready to prove the main result of the section.
\begin{proof}[Proof of Theorem \ref{decay}]
If $u \in I_{k,n}$, then
\[ u= d_1^k \cdot \left(1+ O\left(\frac{1}{d}\right)\right).\]
Using (\ref{tail}) we obtain that there exists  $ C(u) \in [\frac{n_1}{N}, 1]$ such that
\[ \Pv(\deg(V) > u)= \left(\frac{n_1}{N}\right)^k C(u).\]
The last two formulas immediately imply the assertion  of the Theorem whenever $u \in I_{k,n}$. Actually in this case we have $\frac{n_1}{N}\le L(u)\le 1$.  If $u\not \in \cup_{k} I_{n,k}$, then there exists $k=k(u)$ such that $u \in (c_n t_k, c_n t_{k+1})$. By monotonicity of the distribution function we have \[ \Pv(\deg(V)> c_n t_{k+1})\le \Pv(\deg(V)>u)\le \Pv(\deg(V)> c_n t_k).\]
Applying the theorem for $c_n t_{k+1}$ and $c_n t_k$, we loose a factor of $\frac{N_1}{n}$ in the upper bound of $L(u)$ and the assertion of the Theorem follows.
\end{proof}
\begin{comment}

We may also want to see that in the graph $G^\mathrm{r}_n$, there is no vertex which behaves irregularly, i.e. it has significantly more or less neighbors than it code determines. Then we can require for $c_n N^n e^{-\frac18 c_n} $ to be summable. If this is the case, then the graph basically looks similar to its deterministic pair $G_n$ blown up by a factor of $c_n$. For this, $c_n\ge8n\log(N+1)$ is enough.
\end{comment}

\end{document}